\documentclass[a4paper,11pt,reqno]{article}

\usepackage[utf8]{inputenc}
\usepackage[T1]{fontenc}
\usepackage{lmodern}
\usepackage[english]{babel}
\usepackage{microtype}

\usepackage{amsmath,amssymb,amsfonts,amsthm}
\usepackage{mathtools,accents}
\usepackage{mathrsfs}
\usepackage{aliascnt}
\usepackage{braket}
\usepackage{bm}

\usepackage[a4paper,margin=3cm]{geometry}
\usepackage[citecolor=blue,colorlinks]{hyperref}

\usepackage{enumerate}
\usepackage{xcolor}

\makeatletter
\g@addto@macro\@floatboxreset\centering
\makeatother

\makeatletter
\def\newaliasedtheorem#1[#2]#3{
  \newaliascnt{#1@alt}{#2}
  \newtheorem{#1}[#1@alt]{#3}
  \expandafter\newcommand\csname #1@altname\endcsname{#3}
}
\makeatother

\numberwithin{equation}{section}

\theoremstyle{plain}
\newtheorem{theorem}{Theorem}[section]
\newaliasedtheorem{proposition}[theorem]{Proposition}
\newaliasedtheorem{lemma}[theorem]{Lemma}
\newaliasedtheorem{corollary}[theorem]{Corollary}
\newaliasedtheorem{counterexample}[theorem]{Counterexample}

\theoremstyle{definition}
\newaliasedtheorem{definition}[theorem]{Definition}
\newaliasedtheorem{question}[theorem]{Question}
\newaliasedtheorem{example}[theorem]{Example}

\theoremstyle{remark}
\newaliasedtheorem{remark}[theorem]{Remark}

\newcommand{\setN}{\mathbb{N}}
\newcommand{\setQ}{\mathbb{Q}}
\newcommand{\setR}{\mathbb{R}}

\newcommand{\eps}{\varepsilon}

\let\altphi\phi
\let\phi\varphi
\let\varphi\altphi
\let\altphi\undefined

\newcommand{\abs}[1]{\left\lvert#1\right\rvert}
\newcommand{\norm}[1]{\left\lVert#1\right\rVert}

\newcommand{\weakto}{\rightharpoonup}
\newcommand{\Weakto}[1]{\xrightharpoonup{#1}}

\newcommand{\Id}{\mathrm{Id}}

\let\div\undefined
\DeclareMathOperator{\div}{div}
\DeclareMathOperator{\sign}{sign}
\newcommand{\plchldr}{\,\cdot\,}
\newcommand{\di}{\mathop{}\!\mathrm{d}}

\newcommand{\didi}[1]{\frac{\mathrm{d}}{\di#1}}
\newcommand{\bs}{{\rm bs}}
\newcommand{\loc}{{\rm loc}}
\newcommand{\ev}{{\rm ev}}
\newcommand{\sym}{{\rm sym}}
\newcommand{\res}{\mathop{\hbox{\vrule height 7pt width .5pt depth 0pt
\vrule height .5pt width 6pt depth 0pt}}\nolimits}

\newcommand{\restr}{\raisebox{-.1618ex}{$\bigr\rvert$}}
\DeclareMathOperator{\supp}{supp}
\newcommand{\rmD}{{\rm D}}
\newcommand{\Ch}{{\sf Ch}}

\DeclareMathOperator{\Lip}{Lip}
\newcommand{\Lipa}{\operatorname{Lip}_{\rm a}}
\DeclareMathOperator{\Lipb}{Lip_{\rm b}}
\DeclareMathOperator{\Lipbs}{Lip_\bs}
\DeclareMathOperator{\Liploc}{Lip_\loc}
\newcommand{\Cb}{C_{\rm b}}
\newcommand{\Cbs}{C_\bs}
\newcommand{\Cbsx}{C_{\rm bsx}}

\newcommand{\haus}{\mathscr{H}}
\newcommand{\leb}{\mathscr{L}}
\newcommand{\Meas}{\mathscr{M}}
\newcommand{\Prob}{\mathscr{P}}
\newcommand{\Borel}{\mathscr{B}}

\newcommand{\bb}{{\boldsymbol{b}}}
\newcommand{\sbb}{{\boldsymbol{b}}}
\newcommand{\cc}{{\boldsymbol{c}}}

\newcommand{\XX}{{\boldsymbol{X}}}
\newcommand{\ppi}{{\boldsymbol{\pi}}}

\newcommand{\dist}{\mathsf{d}}

\newcommand{\meas}{\mathfrak{m}}

\newcommand{\Algebra}{{\mathscr A}}
\newcommand{\Der}{{\rm Der}}
\newcommand{\Div}{{\rm Div}}
\DeclareMathOperator{\RCD}{RCD}

\begin{document}

\renewcommand{\sectionautorefname}{Section}

\title{Weak and strong convergence of derivations and \\ stability of flows with respect to MGH convergence}
\author{Luigi Ambrosio
\thanks{Scuola Normale Superiore, \url{luigi.ambrosio@sns.it}} \and
Federico Stra
\thanks{Scuola Normale Superiore, \url{federico.stra@sns.it}} \and
Dario Trevisan
\thanks{Universit\`a di Pisa, \url{dario.trevisan@unipi.it}}}
\maketitle

\begin{abstract}
This paper is devoted to the study of weak and strong convergence of derivations, and of the flows associated to them,
when dealing with a sequence of metric measure structures $(X,\dist,\meas_n)$, $\meas_n$ weakly convergent to $\meas$.
In particular, under curvature assumptions, either only on the limit metric structure $(X,\dist,\meas)$ or on the whole sequence of metric
measure spaces, we provide several stability results.
\end{abstract}

\tableofcontents

\section{Introduction}

In this paper we study convergence of vector fields, more precisely derivations, and Lagrangian flows with respect to measured Gromov-Hausdorff 
convergence. The literature on
the topic of MGH convergence is very wide (see for instance  \cite{GigliMondinoSavare13}, \cite{Gromov}, \cite{Sturm06}, \cite{Shioya}, \cite{Villani}) 
and the convergence problems can be attacked from different points of view, depending
also on the richness of the structure under consideration. Our motivations come from two different directions: the first one is a deeper
investigation of the convergence of gradient derivations, particularly in the case when all spaces $(X_n,\dist_n,\meas_n)$ belong to
the same class $\RCD(K,\infty)$ of \cite{AmbrosioGigliSavare14}. In this respect, our results provide a kind of local version of the Mosco convergence of the
$2$-Cheeger energies estabilished in \cite{GigliMondinoSavare13}. This local version will play a role in a forthcoming paper \cite{AmbrosioHonda16}, where Mosco convergence is proved
for all $p$-Cheeger energies, also with stability results for $BV$ functions, Cheeger constants and Hessians.
The second motivation comes as a natural complement of \cite{AmbrosioTrevisan14}, where well-posedness of flows associated to sufficiently
regular derivations $\bb$ (see \autoref{thm:well-posedness-rcd}) 
is proved for the first time in a nonsmooth setting, which includes $\RCD(K,\infty)$ metric measure spaces $(X,\dist,\meas)$. It is
then natural to investigate the stability question (as in the Euclidean theory \cite{AmbrosioCrippa13}), assuming that the derivations depend
on $n$ and that the reference measures $\meas_n$ are variable.

In this paper, when dealing with these problems, we systematically adopt the so-called extrinsic point of view in MGH convergence. Namely,
up to an isometric embedding, we assume that neither $X$ nor $\dist$ depend on $n$. We assume then that $\meas_n$
are locally uniformly finite and nonnegative Borel measures in $(X,\dist)$ which weakly converge to $\meas$, namely
$\int_X v\di\meas_n\to\int_X v\di\meas$ for all $v\in C_\bs(X)$, where $C_\bs(X)$ is the space of bounded continuous
functions with bounded support. 
As illustrated for instance in \cite{GigliMondinoSavare13} (where several notions of convergence are
carefully compared, for sequences of pointed metric measure spaces), this point of view is not really restrictive, and basically unavoidable 
when treating Lagrangian
questions such as the convergence of flows. Indeed, some coupling between points in the different spaces is necessary to prove
convergence of paths to paths, and the simplest way to achieve this is to embed all metric structures into a common one.
Still in connection with the metric structure $(X,\dist)$, we assume only completeness and separability; dropping
local compactness assumptions is useful to include in the theory all $\RCD(K,\infty)$ spaces.

Now we pass to a more detailed description of the content of this paper. It consists of a first part where we recall the basic properties
of derivations and two more parts, an ``Eulerian'' one, dealing with the convergence of derivations and particularly of gradient derivations,
and a ``Lagrangian'' part, dealing with the convergence of flows, that can be read almost independently.

\paragraph{\autoref{pt:preliminary}.} In the seminal paper \cite{Weaver} the concept of derivation is used to build a good notion of tangent
bundle in metric spaces, the main idea being to describe the bundle implicitly through the collection of its ($\meas$-measurable) sections.
These sections, called derivations, are linear maps $\bb:\Lipb(X)\to L^0(X,\meas)$ which satisfy suitable continuity and locality
properties. In this paper we almost completely adopt the point of view of \cite{Gigli}, where besides linearity one assumes the validity
of the inequality
\begin{equation}\label{eq:intro0}
|\bb(f)|\leq h|\rmD f|\qquad\text{$\meas$-a.e.~in $X$, for all $f\in\Lipb(X)$.}
\end{equation}
The only difference is that, since we are dealing with a sequence of measures, we keep $\Lipb(X)$ as
domain of derivations, to have $\meas$-independent domains, instead of the Sobolev space considered in  \cite{Gigli}.
In \cite{Gigli} a systematic analysis of first and second order calculus based on this notion is made, including in particular
the $L^2$ duality between tangent and cotangent bundle, see \autoref{rdual}.  The quantity $|\rmD f|$ in the right hand side of \eqref{eq:intro0} 
is the minimal relaxed slope of \cite{Cheeger}, which provides integral representation to
Cheeger's energy $\Ch$, see \autoref{appendix:slopes} for a quick introduction to this concept.
By duality, the minimal $h$ in \eqref{eq:intro0} is denoted by $|\bb|$ (strictly speaking, we should
use the notation $\Ch_\meas$, $|\bb|_\meas$ to stress the dependence of these concepts on $\meas$). The $L^2$ duality estabilished in \cite{Gigli} is
particularly useful for us to read the Hilbertian character of the norm $|\bb|$ in terms of the cotangent bundle,
i.e.~in terms of quadraticity of Cheeger's energy. According to the theory developed in \cite{AmbrosioGigliSavare13} and \cite{AmbrosioGigliSavare14}, this is
the analytically most convenient formulation. 

In order to treat weak convergence of derivations from the sequential point of view it is technically useful to consider a countable 
algebra $\Algebra\subset\Lipb(X)$ of ``test'' functions. In the spirit of Gromov's reconstruction theorem \cite{Gromov} and many other
results of the theory (see also \cite{Keith} in connection with Cheeger's coordinates in PI spaces and the more recent work 
\cite{CheegerKleinerSchioppa}) it is natural to consider the
algebra finitely generated by truncations of distance functions. To make it separable, we consider a countable dense set $D\subset X$ and define
$\Algebra$ as in \eqref{eq:setD}, also considering the subalgebra $\Algebra_\bs$ of functions with bounded support. In connection
with this choice, denoting by $H^{1,2}(X,\dist,\meas)$ the domain of $\Ch$, we shall also need this approximation result
(where $\Lipa $ stands for the asymptotic Lipschitz constant, see \eqref{eq:deflipa}): 
\textit{for all $f$ in $H^{1,2}(X,\dist,\meas)$, there exist $f_n\in\Algebra_\bs$
with $f_n\to f$ and $\Lipa (f_n)\to |\rmD f|$ in $L^2(X,\meas)$.} 

Its proof, given in \autoref{appendix:approximation}, is a further refinement of the techniques and of the results developed in 
\cite{AmbrosioGigliSavare13} (see also \cite{AmbrosioColomboDiMarino}), where the approximating functions where chosen, as in
\cite{Cheeger}, in the larger class $\Lipb(X)\cap L^2(X,\meas)$.

\paragraph{\autoref{pt:derivations}.} In this part we discuss weak convergence of derivations and criteria for strong convergence. Even though some of our
results deal with time-dependent derivations (more natural for the study of flows, in \autoref{pt:flows}), we present in this introduction only
the autonomous case. We say that $\bb_n$, derivations in $(X,\dist,\meas_n)$, {\it weakly} converge in duality with
$\Algebra_\bs$ to $\bb$, derivation in $(X,\dist,\meas)$, and write $\bb_n\Weakto{\Algebra_\bs}\bb$, if 
\begin{equation}\label{eq:firstweak}
\lim_{n\to\infty} \int_X \bb_n(f)v\di\meas_n=\int_X\bb(f)v\di\meas\qquad\forall f\in\Algebra_\bs,\,\,v\in\Cb(X).
\end{equation}

Recall that the divergence $\div\bb$ of a derivation is defined by duality, via the formula
$\int_X f\div\bb\di\meas=-\int_X\bb(f)\di\meas$ for all $f\in\Lip_\bs(X)$
(as for $\Ch$ and $|\bb|$, we should use the notation $\div_{\meas}$, to stress its dependence on $\meas$).
In presence of uniform bounds on $|\bb_n|$ and on their divergences, \autoref{thm:compactness} provides a sequential compactness
result totally independent of regularity assumptions on the metric measure structures. 
Due to the close relation between vector fields with bounds on divergence and normal currents (see in particular the
discussion in \cite[Appendix~A]{PaoliniStepanov}), this result is reminiscent
of the compactness result for normal currents in \cite{Ambrosio-Kirchheim}. Notice that the above mentioned approximation result, 
provided in \autoref{appendix:approximation},  is necessary to pass to
the limit in the inequality $|\bb_n(f)|\leq |\bb_n|_{\meas_n}|\rmD f|_{\meas_n}$, using as intermediate step the weaker inequality with the asymptotic
Lipschitz constant in the right hand side. 

When we deal with gradient derivations the bounds on divergence correspond to Laplacian bounds, which could be too restrictive for
some applications. For this reason, we exploit the regularizing properties of the heat semigroup $P_t$.  
Indeed, under suitable regularity assumptions (only) on the limit metric measure structure $(X,\dist,\meas)$,
which are satisfied by all $\RCD(K,\infty)$ spaces, it is technically convenient to replace in \eqref{eq:firstweak} the algebra $\Algebra_\bs$ by a more regular 
class of test functions, namely $P_{\setQ_+}\Algebra_\bs := \{P_t f:\ t\in\setQ_+,\,\,f\in\Algebra_\bs\}$, letting $v$ vary in $C_\bs(X)$ (since the support of
$P_tf$ may have infinite measure, unless all measures $\meas_n$ are finite). With this new notion of weak convergence, for which
we use the notation $\bb_n\Weakto{P\Algebra_\bs}\bb$, we provide another
sequential compactness result, see \autoref{thm:compa2}, free of divergence bounds;  in 
\autoref{rem:compaweak} we show that, under uniform bounds on divergence, the two notions of convergence are equivalent.

In addition, for the convergence in duality with $P_{\setQ_+}\Algebra_\bs$, we provide a criterion for strong $L^p$ convergence, namely
\begin{equation}\label{eq:firststrong}
\lim_{n\to\infty} \int_X |\bb_n(f)|^p\di\meas_n\leq\int_X|\bb(f)|^p\di\meas\qquad\forall f\in P_{\setQ_+}\Algebra_\bs,
\end{equation}
under the assumption that there is no loss of norm in the limit, namely
$$
\limsup_{n\to\infty}\int_X|\bb_n|^p\di\meas_n\leq\int_X|\bb|^p\di\meas<\infty.
$$
The proof of this criterion, given in \autoref{thm:strong-stability-hilbert}, is probably the most technical part of this paper: its proof follows closely
ideas from the theory of Young measures (\cite{Valadier}, \cite[Section~5.4]{AmbrosioGigliSavare08}), with the extra difficulties due to the fact that 
we do not have a pointwise description of the tangent bundle (but, via a suitable concept of pre-derivation, we provide a kind of replacement for it).

Finally, we conclude this part by discussing convergence of gradient derivations. If $f_n\in H^{1,2}(X,\dist,\meas_n)$
strongly converge in $H^{1,2}$ to $f\in H^{1,2}(X,\dist,\meas)$ (in particular $\limsup_n\Ch_n(f_n)\leq\Ch(f)$)  
we prove in \autoref{thm:strong-convergence-gradients}, under a Mosco convergence assumption that the induced derivations $\bb_{f_n}$ strongly converge in $L^2$ to $\bb_f$, 
meaning that $\bb_{f_n}(a)$ converge in $L^2$ to $\bb_f(a)$ for all  $a\in P_{\setQ_+}\Algebra_\bs$. The Mosco convergence  assumption is fulfilled by sequences of $\RCD(K,\infty)$ spaces, as in \cite{GigliMondinoSavare13}; these results will play an important 
role in \cite{AmbrosioHonda16}.

\paragraph{\autoref{pt:flows}.} In this part we discuss the stability of flows w.r.t.~strong convergence of derivations. As in \cite{AmbrosioTrevisan14}, we say that
$\XX(t,x)$ is a regular flow relative to a possibly time dependent derivation $\bb_t$, $t\in (0,T)$, if $\XX(0,x)=x$, $\XX(\plchldr,x)$ is absolutely continuous
in $[0,T]$ for $\meas$-a.e.~$x$ and solves the ODE $\gamma'=\bb_t\circ\gamma$ in the following weak sense:
\begin{equation}\label{eq:intro3}
\didi t f\circ\XX(t,x)=\bb_t(f)(\XX(t,x))\quad\text{for $\leb^1\times\meas$-a.e.~$(t,x)\in (0,T)\times X$,} 
\end{equation}
for all $f\in\Lipb(X)$. The adjective ``regular'' refers, as in the Euclidean theory (see \cite{AmbrosioCrippa13} and the references therein) to the non-concentration
condition $\XX(t,\plchldr)_\#\meas\leq C\meas$ for all $t\in (0,T)$, with $C=C(X,\meas)$. When the limit structure $(X,\dist,\meas)$ is
a $\RCD(K,\infty)$ space and  the regularity assumptions on $\bb$ of 
\autoref{thm:well-posedness-rcd} hold (which ensure uniqueness of the regular flow $\XX$ relative to $\bb_t$), we are 
able to provide in \autoref{thm:stabflow_curv} a convergence result for regular flows $\XX_n$ relative to $\bb_{n,t}$, assuming strong convergence of
$\bb_{n,t}$ to $\bb_t$, uniform growth bounds and $\sup_nC(\XX_n,\meas_n)<\infty$. Here, convergence is undertood as convergence in measure,
namely convergence of the $C([0,T];X)$-valued maps $x\mapsto\XX_n(\plchldr,x)$ to the map $x\mapsto\XX(\plchldr,x)$; the notion of convergence in measure
can be adapted to our case, where even the reference measures are variable.

The strategy is to prove, via tightness estimates, convergence to a 
regular generalized flow $\ppi$ relative to $\bb_t$ (see \autoref{dregflow} for this more general concept of flow) and then use the regularity
of $\bb_t$ to extract a ``deterministic'' flow out of it. The key step, where strong convergence of the derivations $\bb_{n,t}$ is involved, is that the ODE
condition \eqref{eq:intro3} passes to the limit; in the proof of this we use a new principle (see \autoref{prop:conc}) based on the continuity equation
which would lead also to the simplification of some proofs of the Euclidean theory (for instance,  \cite[Theorem~12]{AmbrosioCrippa13}).

\paragraph{Acknowledgments.}
The third author has been supported by project PRA\_2016\_41 from Pisa University. All authors are members of the GNAMPA research
group of the Istituto Nazionale di Alta Matematica (INdAM).

\part{Preliminary results}\label{pt:preliminary}

\section{Notation}

\paragraph{Metric concepts.} In a metric space $(X,\dist)$, we denote by $B_r(x)$ and $\bar{B}_r(x)$ the open and closed
balls respectively, by $\Cbs(X)$ the space of continuous functions with bounded support, by $\Lip_\bs(X)\subset \Cbs(X)$  the subspace of Lipschitz 
functions. We use the notation $\Cb(X)$ and $\Lipb(X)$ for bounded continuous and bounded Lipschitz functions respectively.

For a function $f:X\to\setR$ we denote by $\Lip(f)\in [0,\infty]$ its Lipschitz constant. We also define the asymptotic Lipschitz constant by
\begin{equation}\label{eq:deflipa}
\Lipa f(x) = \inf_{r>0} \Lip \left( f\restr_{B_r(x)} \right) =
\lim_{r\to0^+} \Lip \left( f\restr_{B_r(x)} \right),
\end{equation}
which is upper semicontinuous.
If $I\subset\setR$ is an interval, we denote by $AC(I;X)\subset C(I;X)$ the 
space of absolutely continuous maps w.r.t.~$\dist$, satisfying
\begin{equation}\label{eq:minimalg}
\dist(\gamma(s),\gamma(t))\leq\int_s^tg(r)\di r\qquad\text{for all $s,\,t\in I$ with $s\leq t$}
\end{equation}
for some (nonnegative) $g\in L^1(I)$. We denote by $\ev_t:C(I;X)\to X$ the evaluation map at time $t$, i.e.~$\ev_t(\gamma)=\gamma(t)$.

The metric derivative $|\dot\gamma|:\ring{I}\to [0,\infty]$ of $\gamma\in AC(I;X)$ is the Borel map defined by
$$
|\dot\gamma|(t):=\limsup_{s\to t}\frac{\dist(\gamma(s),\gamma(t))}{|s-t|}
$$
and it can be proved (see for instance \cite[Theorem~1.1.2]{AmbrosioGigliSavare08}) that the $\limsup$ is a limit $\leb^1$-a.e.~in $I$, 
that $|\dot\gamma|\in L^1(I)$ and that $|\dot\gamma|$ 
is the smallest $L^1$ function with the property \eqref{eq:minimalg}, up to $\leb^1$-negligible sets.

\paragraph{The metric algebras $\Algebra$, $\Algebra_\bs$.} In the sequel we associate to any separable metric space $(X,\dist)$ a countable dense set $D \subset X$ and the smallest 
set $\Algebra\subset\Lipb(X)$ containing
\begin{equation}\label{eq:setD}
\min\{\dist(\plchldr,x),k\}\quad\text{with $k\in\setQ\cap [0,\infty]$, $x\in D$,}
\end{equation}
which is a vector space over $\setQ$ and is stable under products and lattice operations.
It is a countable set and it depends only on the choice of $D$ (we do not emphasize this dependence 
in our notation). We shall also consider the subalgebra $\Algebra_\bs$ of functions with bounded support.

\paragraph{Measure-theoretic notation.} The Borel $\sigma$-algebra of a metric space $(X,\dist)$ is denoted by $\Borel(X)$.
The Borel signed measures with finite total variation are denoted by $\Meas(X)$, while we use the notation $\Meas^+(X)$, $\Meas^+_\loc(X)$, $\Prob(X)$ for 
nonnegative finite Borel measures, Borel measures which are finite on bounded sets and Borel probability measures. For $E\in\Borel(X)$
the restriction operator $\mu\mapsto\mu\res E$ is defined by $\mu\res E(B)=\mu(B\cap E)$ for all $B\in\Borel(X)$.

We use the standard notation $L^p(X,\meas)$, $L^p_\loc(X,\meas)$ for the $L^p$ spaces
when $\meas$ is nonnegative, including also the case $p<1$ (when $p=0$ the spaces $L^0(X,\meas)=L^0_\loc(X,\meas)$ correspond to the class of
real-valued $\meas$-measurable functions). Notice that, in this context where no local compactness assumption is made, $L^p_\loc$ means
$p$-integrability on bounded subsets.

Given metric spaces $(X,\dist_X)$ and $(Y,\dist_Y)$ and a Borel map $f:X\to Y$, we denote by $f_\#$ the induced push-forward
operator, mapping $\Prob(X)$ to $\Prob(Y)$, $\Meas^+(X)$ to $\Meas^+(Y)$ and, if the preimage of bounded sets is bounded, 
$\Meas^+_\loc(X)$ to $\Meas^+_\loc(Y)$. Notice that, for all $\mu\in\Meas^+(X)$, $f_\#\mu$ is well defined also if $f$ is $\mu$-measurable.

\begin{definition}[Metric measure space]
A \emph{metric measure space} is a triple $(X,\dist,\meas)$, where $(X,\dist)$
is a complete and separable metric space and $\meas\in\Meas_\loc^+(X)$. 
\end{definition}

\paragraph{Convergence of functions and measures.} We say that $f_n\in\Lipb(X)$ converge in the \emph{flat sense} to $f\in\Lipb(X)$ if $f_n\to f$ pointwise
in $X$ and $\sup_n(\sup|f_n|+\Lip(f_n))<\infty$. If pointwise convergence occurs only $\meas$-a.e., we say that $f_n\to f$ in the
$\meas$\emph{-flat sense}.

We say that $\meas_n\in\Meas_\loc(X)$ weakly converge to $\meas\in\Meas_\loc(X)$ if
$\int_X v\di\meas_n\to\int_X v\di\meas$ as $n\to\infty$ for all $v\in \Cbs(X)$. When all the measures $\meas_n$ as well as $\meas$
are probability measures, this is equivalent to requiring that $\int_X v\di\meas_n\to\int_X v\di\meas$ as $n\to\infty$ for all $v\in\Cb(X)$. 
We shall also use the following well-known proposition.

\begin{proposition}\label{prop:passa_al_limite} If $\meas_n$ weakly converge to $\meas$ in $\Meas^+_\loc(X)$ and if
$$
\limsup_{n\to\infty}\int_X\Theta\di\meas_n<\infty
$$
for some Borel $\Theta:X\to [0,\infty]$, then 
\begin{equation}\label{eq:passa_al_limite}
\lim_{n\to\infty}\int_X v\di\meas_n=\int_X v\di\meas
\end{equation} 
for all $v:X\to\setR$ continuous with $\lim_{\dist(x,\bar x)\to\infty}|v|(x)/\Theta(x)=0$
for some (and thus all) $\bar x\in X$. If $\Theta:X\to [0,\infty)$ is continuous and
$$\limsup_{n\to\infty}\int_X\Theta\di\meas_n\leq\int_X\Theta\di\meas<\infty,$$ 
then \eqref{eq:passa_al_limite} holds for all $v:X\to\setR$ continuous with
$|v|\leq C\Theta$ for some constant $C$.
\end{proposition}

For a sequence $(\meas_n)\subset\Meas_\loc^+(X)$ that weakly converges to $\meas\in\Meas^+_\loc(X)$, we introduce, following
the presentation in \cite{GigliMondinoSavare13}, notions of weak and strong convergence for a sequence of Borel functions 
$f_n: X \to\setR$, $n \ge 1$.  It is difficult to trace when these concepts, very natural when variable reference measures
are involved, have been introduced for the first time; good references in geometric setups are \cite{Hutchinson} and
\cite{Kuwae-Shioya}.

\paragraph{Weak convergence.} For $p \in (1, \infty]$, we say that $f_n \in L^p(X,\meas_n)$ converge to $f \in L^p(X,\meas)$ weakly 
(weakly-* if $p =\infty$) in $L^p$ if $f_n \meas_n$ weakly converge to $f \meas$ in $\Meas_\loc (X)$, with
\begin{equation}\label{eq:bound-lp-norms} 
\limsup_{n\to \infty} \left\| f_n \right\|_{L^p(X,\meas_n)} < \infty.
\end{equation}
In such a situation, one has $\left\| f \right\|_{L^p(X,\meas)} \le  \liminf_n \left\| f_n \right\|_{L^p(X,\meas_n)}$. 
Moreover, if $p\in (1,\infty)$, any sequence $f_n \in L^p(X,\meas_n)$ such that \eqref{eq:bound-lp-norms} holds
admits a subsequence weakly converging to some $f \in L^p(X,\meas)$.

We notice that, if two sequences $f_n$, $g_n \in L^p(X,\meas_n)$ weakly converge in $L^p$ respectively to $f$, $g \in L^p(X,\meas)$ and satisfy $0 \le f_n \le g_n h$, $\meas_n$-a.e.~in $X$, 
where $h: X\to\setR^+$ is bounded and upper semicontinuous, then $f \le g h$, $\meas$-a.e.~in $X$. This follows from duality with $\phi \in \Cbs(X)$, $\phi \ge 0$, and the fact that $h$ is the pointwise infimum of bounded continuous functions. 
 
\paragraph{Strong convergence.} For $p \in (1, \infty)$ we say that $f_n\in L^p(X,\meas_n)$ converge to $f$ strongly in $L^p$ if, in addition to weak convergence in $L^p$, one has 
\begin{equation}\label{eq:kadec}
\limsup_n\left\| f_n \right\|_{L^p(X,\meas_n)} \le \left\| f \right\|_{L^p(X,\meas)}.
\end{equation}  
This condition, in the standard setting of a \textit{fixed} Banach space, is reminiscent of the Kadec-Klee property, 
which yields strong convergence out of weak convergence plus convergence of norms. 
In our more general situation, one can prove that, under \eqref{eq:kadec}, weak convergence of $f_n \meas_n \in \Meas_\loc (X)$ to $f\meas$ improves to 
\begin{equation} \label{eq:duality-strong-weak-convergence}
\lim_{n \to \infty} \int_X v_n f_n \meas_n  = \int_X v f \meas, \quad \text{$\forall v_n \in L^{p'}(X,\meas_n)$ weakly converging in $L^{p'}$ to $v$,}\end{equation}
where $p'$ is the conjugate exponent of $p$. Moreover, given two sequences $f_n,\,g_n \in L^p(X,\meas_n)$ strongly converging in $L^p$
respectively to $f,\,g\in L^p(X,\meas)$, one can show that $f_n+g_n$ strongly converges in $L^p$ to $f+g$. Other improvements,
which also motivate the terminology strong convergence, are discussed in \autoref{rem:sono_simili}.

In the special case $p=1$, we say that $f_n \in L^1(X,\meas_n)$ converge to $f \in L^1(X,\meas_n)$ weakly 
in $L^1$ if $f_n \meas_n$ weakly converge to $f \meas$ in $\Meas_\loc (X)$. Since in the context of variable measures
it seems not easy to formulate an adequate notion of equi-integrability, this notion is strictly weaker than the classical
weak $L^1$ convergence when the measure is fixed. However, sufficient conditions for sequential compactness
w.r.t.~weak $L^1$ convergence can be stated also in the case of variable measures. 

\begin{lemma}\label{lem:DPett}
The following properties hold.
\begin{itemize} 
\item[(i)] 
If $|f_n|\leq g_n h_n$ with $g_n$ bounded in $L^2$ and $h_n$ strongly convergent in $L^2$, then $(f_n)$ has subsequences $(f_{n(k)})$ weakly convergent in $L^1$.
\item[(ii)]
If $\sup_n\int_X\Theta(|f_n|)\di\meas_n<\infty$ for some $\Theta:[0,\infty)\to [0,\infty]$ with more than linear growth at infinity, then
$(f_n)$ has subsequences $(f_{n(k)})$ weakly convergent in $L^1$. In addition, if $\Theta$ is convex and lower semicontinuous,
$$
\int_X\Theta(|f|)\di\meas\leq\liminf_{k\to\infty}\int_X\Theta(|f_{n(k)}|)\di\meas_{n(k)}.
$$
\end{itemize}
\end{lemma}
\begin{proof}~(i) Let $h$ be the strong $L^2$ limit of $h_n$ and $L=\sup_n\|g_n\|_{L^2(X,\meas_n)}$.
It is easy to check that any limit point $\mu$ of $f_n\meas_n$ satisfies $\mu(C)\leq L\bigl(\int_Ch^2\di\meas)^{1/2}$
for all $C\subset X$ closed and bounded, hence $\mu\ll\meas$.

(ii) It follows by a classical duality argument, see for instance \cite[Lemma~9.4.3]{AmbrosioGigliSavare08}.
\end{proof}

%
%
All such notions of weak and strong convergence have natural local analogues. In particular, we say that $f_n$ converge to $f$ 
weakly in $L^p_\loc$ (resp. weakly-* if $p =\infty$) if $f_nv$ converge to $fv$ weakly in $L^p$ (resp. weakly-* if $p=\infty$)
for all $v\in C_\bs(X)$. In the case of strong $L^p_\loc$ convergence, $p\in (1,\infty)$, we ask that $f_n \meas_n$ weakly converge 
to $f \meas$ in $\Meas_\loc (X)$, with
\begin{equation*}
\limsup_{n\to \infty} \left\| f_n \right\|_{L^p(B_R(\bar x),\meas_n)}\leq
\|f\|_{L^p(B_R(\bar x),\meas)}<\infty
\end{equation*}
for some $\bar x\in X$ and arbitrarily large $R$. It follows immediately that \eqref{eq:duality-strong-weak-convergence} holds 
for all $v_n$  weakly converging in $L^{p'}$ to $v$ with $\cup_n\supp v_n$ bounded in $X$.


Even in the simplest case when $\meas_n=\meas$, another natural notion of convergence for metric-valued maps is 
convergence in measure.

\paragraph{Convergence in measure for metric space valued maps.} Given a metric space $(Y,\dist_Y)$, 
we say that Borel functions $f_n: X \to Y$ (locally) converge in measure to a Borel $f: X \to Y$ if, for every $\Phi\in\Cb(Y)$ one has that
$\Phi(f_n)\meas_n$ weakly converge to $\Phi(f)\meas$, namely
\begin{equation}\label{eq:weakPhi} 
\lim_{n \to \infty} \int_X v\Phi(f_n)\di \meas_n  = \int_X v \Phi(f)\di \meas\qquad\forall v\in\Cbs(X).
\end{equation}
Notice that, choosing $\Phi$ to be constant, already \eqref{eq:weakPhi} encodes the
weak convergence of $\meas_n$ to $\meas$, so strictly speaking this should be understood as a convergence of pairs $(\meas_n, f_n)$,
where $\meas_n$ are reference measures for $f_n$.
It is also easy to check that, when $\meas_n =\meas$ is fixed, this requirement is equivalent to the usual notion of local convergence in 
$\meas$-measure. 

In the following proposition we provide equivalent formulations of the convergence in measure, closely
related to the theory of Young measures \cite{Valadier}. To this aim, we introduce the following
terminology: we say that a family $(\mu_n)\subset\Meas^+_\loc(X\times Y)$ is locally tight w.r.t.~$X$ if, for some $\bar x\in X$ and
all $R>0$, the measures $\mu_n\res B_R(\bar x)\times Y$ are tight. Accordingly, we denote by $\Cbsx(X\times Y)$ the space
$$
\Cbsx(X\times Y)\coloneqq 
\left\{f\in \Cb(X\times Y) : \supp(f)\subset B_R(\bar x)\times F
\ \text{for some $\bar x\in X$, $R>0$}\right\}
$$
and we say that $\mu_n\Weakto{\Cbsx(X\times Y)}\mu$ if $\mu_n\to\mu$ in the duality with $\Cbsx(X\times Y)$. 

\begin{proposition}\label{prop:convergence-in-measure}
Let $(Y,\dist_Y)$ be a metric space, let $\meas_n$ be weakly convergent to $\meas$ in $\Meas^+_\loc(X)$ 
and let $f_n,\,f: X\to Y$  be Borel. The following conditions are equivalent:

\begin{itemize}
\item[(a)] $f_n$ converge to $f$ in measure;
\item[(b)]  for any $\Phi \in \Cb(Y)$ and any $v_n\in L^1(X,\meas_n)$ weakly converging to $v\in L^1(X,\meas)$ in $L^1$ with
$\cup_n\supp(v_n)\subset B_R(\bar x)$ for some $\bar x\in X$ and $R>0$, one has
\[ \lim_{n \to \infty} \int_X v_n \Phi(f_n)\di \meas_n  = \int_X v \Phi(f)\di \meas;\] 
\item[(c)] $(\Id \times (\Phi\circ f_n))_{\#} \meas_n$ converge in duality with $\Cbsx(X\times\setR)$ to 
$(\Id \times (\Phi \circ f))_\#\meas$ for any $\Phi \in \Cb(Y)$.
\end{itemize}
Under the assumption that $(\Id \times f_n)_\#\meas_n$ are locally tight w.r.t.~$X$, (a), (b), (c) are also equivalent to
$(\Id \times f_n)_\#\meas_n\Weakto{\Cbsx(X\times Y)}(\Id \times f)_\# \meas$.
\end{proposition}

\begin{proof} For ease of notation, write $\mu_n := (\Id \times f_n)_\#\meas_n$, and $\mu = (\Id \times f)_\#\meas$.
It is obvious that the convergence of $\mu_n$ to $\mu$ in duality with $\Cbsx(X\times Y)$ implies (c), and that (b) implies (a).
 
We first prove the implications (c)$\implies$(b) and (a)$\implies$(c), which provide the equivalence between (a), (b), (c). 
In the proof of these two implications,  without any loss of generality we may assume that $Y$ is a compact interval of $\setR$, 
say $Y = [0,1]$, and that $\Phi$ is the identity map (it is sufficient to argue with fixed $\Phi \in \Cb(X)$ and replace $f_n$ with  $\Phi \circ f_n$, $f$ with $\Phi \circ f$).

(c)$\implies$(b). Let $v_n$, $v$ as in (b), with $\sup_n\norm{ v_n }_{L^1(X,\meas_n)} \le 1$, and let $w\in C_\bs(X)$. 
By weak convergence of the measures $v_n \meas_n$ to $v \meas$,
\[
 \lim_{n \to \infty} \int_X v_n f_n \di \meas_n = \lim_{n \to \infty} \int_X v_n w \di \meas_n + \int_X v_n (f_n - w)\di \meas_n
=  \int v w \di \meas,
\]
because
\begin{equation*}\begin{split}
\limsup_{n\to \infty} & \left| \int_X  v_n (f_n - w)\di \meas_n \right | 
  = \limsup_{n\to \infty} \left| \int_{X\times Y} v_n(x) \left(y - w(x)\right) \di \mu_n (x,y)\right | \\
  & \le\limsup_{n\to \infty} \left\{M \int_{B_R(\bar x)\cap \{ \abs{v_n} \le M \}} \left| y - w(x) \right | \di \mu_n(x,y) + \meas_n \left(\{\abs{v_n} > M \}\right)\right\} \\  
  & \le   M \int_{\bar{B}_R(\bar x)\times Y} \left| y - w(x) \right | \di \mu(x,y) + \frac 1 M \\
 &=  M \int_{\bar{B}_R(\bar x)}\left| { f(x) - w(x)} \right | \di \meas(x) + \frac 1 M,
\end{split}\end{equation*}
and $M>0$ is arbitrary. By density of $\Cbs(X)$ in $L^1_\loc(\bar{B}_R(\bar x),\meas)$, we 
may optimize the choice of $w$ and then let $M \to \infty$ to conclude.

(a)$\implies$(c).  The family $(\mu_n)$ is locally tight w.r.t.~$X$, thanks to the local tightness of $\meas_n$ in $X$
(recall that we are assuming that $Y$ is compact). In addition, any weak limit point $\sigma$
of $(\mu_n)$ in duality with $\Cbsx(X\times Y)$ agrees with $\mu$ on test functions $\varphi(x,y)$ of the form $\psi(x)\Phi(y)$, 
with $\psi\in\Cbs(X)$ and $\Phi\in\Cb(Y)$. This class of functions is sufficient to prove that $\sigma=\mu$, hence $\mu_n$ weakly converge to 
$\mu$ in duality with $\Cbsx(X\times Y)$.

Finally, for general metric spaces $(Y,\dist_Y)$, if we assume that $(\mu_n)$ is locally tight w.r.t.~$X$, the argument used to prove that (a) implies (c) works and provides the weak convergence of $\mu_n$ to $\mu$ in duality with $\Cbsx(X\times Y)$.
\end{proof}

\begin{remark}[The case of probability measures]\label{rem:measure_when_prob}
In the study of flows, we shall need convergence in measure of $C([0,T];X)$-valued maps; in this case 
the reference measures for our maps will be probability measures (actually marginal measures of probabilities defined on $C([0,T];X)$). In this
simpler case the results of \autoref{prop:convergence-in-measure} can be strengthened.
Indeed when all $\meas_n$ as well as $\meas$ are probability measures, an equivalent formulation of the convergence in measure \eqref{eq:weakPhi}
is with test functions $v\in\Cb(X)$. In addition, stating (b) without any bound on the support of $v_n$ and replacing $\Cbsx(X\times Y)$
with $\Cb(X\times Y)$, in this modified form (a), (b), (c) are still equivalent; finally, assuming a priori the tightness of $\mu_n$,
these notions are equivalent to the weak convergence of $\mu_n$ to $\mu$ in $\Prob(X\times Y)$.
\end{remark}

In the following remark we compare strong convergence with convergence in measure. 

\begin{remark}[Convergence in measure versus strong $L^p$ convergence]\label{rem:sono_simili}
{\rm In the case $Y=\setR$ it is natural to compare strong convergence in $L^p$, $p\in (1,\infty)$, with convergence in measure. 
As in the case of a fixed measure, for sequences convergent in measure and uniformly
bounded in $L^p_\loc$ as in \eqref{eq:bound-lp-norms} for some $p\in (1,\infty)$, one can use the stability of convergence in measure under truncations and property
(b) to show strong convergence in $L^q_\loc$ for any $q\in (1,p)$ and then (thanks to weak $L^p$ compactness), weak convergence in $L^p$ if the
bound is uniform in $L^p$. Conversely,
if $f_n$ converge strongly in $L^p$, strong convexity of the function $\Theta(z)=|z|^p$ and a detailed analysis of the limit points of the measures 
$\mu = (\Id\times f_n)_\#\meas_n$ in the duality with $\Cbsx(X\times\setR)$ 
(see for instance \cite[Section~5.4]{AmbrosioGigliSavare08}, \cite{GigliMondinoSavare13}) show that $f_n\to f$ in measure. 
}\end{remark}

\paragraph{Normed $L^\infty$ modules.} Recall that a (real) $L^\infty(X,\meas)$-module $M$ is a real vector space with the additional structure of bilinear multiplication
by $L^\infty(X,\meas)$ functions $\chi: m\in M\mapsto\chi m\in M$, with the associativity property $\chi(\chi' m)=(\chi\chi') m$;
in addition, multiplication by $\lambda\in\setR$ corresponds to the multiplication by the $L^\infty(X,\meas)$ function equal $\meas$-a.e.~to $\lambda$.

We say that $L^\infty(X,\meas)$-module $M$ is a $L^2(X,\meas)$-normed module if there exists a ``pointwise norm''
$\abs{\plchldr}:M\to \{f\in L^2(X,\meas):\ f\geq 0\}$ satisfying:
\begin{itemize}
\item[(a)] $|m+m'|\leq |m|+|m'|$ $\meas$-a.e.~in $X$ for all $m,\,m'\in M$;
\item[(b)] $|\chi m|=|\chi||m|$ $\meas$-a.e.~in $X$ for all $m\in M$, $\chi\in L^\infty(X,\meas)$;
\end{itemize}
with the property that
\begin{equation}\label{eq:natural_norm}
\|m\| \coloneqq \biggl(\int_X|m(x)|^2\di\meas(x)\biggr)^{1/2}
\end{equation}
is a norm in $M$. Notice that the homogeneity
 and the subadditivity of $\norm{\plchldr}$ are obvious consequences of (a), (b).

If $M$ is a $L^2(X,\meas)$-normed module, we can define $M^*$ as the set of all
continuous $L^\infty(X,\meas)$-module isomorphisms $L:M\to L^1(X,\meas)$, i.e.~linear maps
satisfying, for some $h\in L^2(X,\meas)$ and all $\chi\in L^\infty(X,\meas)$, $m\in M$, the properties
$$
L(\chi m)=\chi L(m),
$$
$$
|L(m)|\leq h|m|\qquad\text{$\meas$-a.e.~in $X$.}
$$
The set $M^*$ can also be given the structure of $L^2(X,\meas)$-module, defining $|L|$ as the least
function $h$ (in the $\meas$-a.e.~sense) satisfying the inequality above.

Finally, we say that a $L^2(X,\meas)$-normed module $M$ is a Hilbert module if $(M,\norm{\plchldr})$ is a Hilbert
space when endowed with the natural norm in \eqref{eq:natural_norm}.

The following result has been proved in \cite{Gigli}; see Proposition~1.2.21 and Theorem~1.2.24 therein.

\begin{theorem}\label{tgigli} Let $(X,\meas)$ be a measure space
and let $M$ be a $L^2(X,\meas)$-normed Hilbert module. Then the following facts hold.
\begin{itemize}
\item[(i)] The map
$$
\langle m,m'\rangle_x:=\frac{|m+m'|(x)^2-|m-m'|(x)^2}{4}
$$
provides a $L^\infty(X,\meas)$-bilinear and continuous map from $M\times M$ to $L^1(X,\meas)$, with 
$\||m|\|^2_{L^2(X,\meas)}=\|\langle m,m\rangle\|_{L^1(X,\meas)}$.
\item[(ii)] (Riesz theorem) For all $L\in M^*$ there exists $m_L\in M$ such that $|L|=|m_L|$ $\meas$-a.e.~in $X$
(in particular $\|L\|=\|m_L\|$) and
$$
L(m)(x)=\langle m_L,m\rangle_x\qquad\text{for $\meas$-a.e.~$x\in X$, for all $m\in M$.}
$$
\end{itemize}
\end{theorem}

\section{Derivations}

The following definition is inspired by Weaver \cite{Weaver}, but in the statement of the quantitative locality property we follow more closely
the point of view systematically pursued in \cite{Gigli}. Recall that $|\rmD f|$ denotes the minimal relaxed slope,
as defined in \autoref{appendix:slopes}. 

\begin{definition}[Derivations]
A \emph{derivation} on a metric measure space $(X,\dist,\meas)$ is a
linear functional $\bb: \Lipb(X) \to L^0(X,\meas)$ which, for some $h\in L^0(X,\meas)$, satisfies the quantitative locality property
$$\abs{\bb(f)} \le h |\rmD f|\qquad\text{$\meas$-a.e.~in $X$, for all $f\in\Lipb(X)$.}$$ 
The $\meas$-a.e.~smallest function $h$ with this property is denoted by $\abs\bb$ and is called (local) norm of $\bb$.
We denote by $\Der(X,\dist,\meas)$ the space of derivations.
For $p\in[1,\infty]$, we denote by $\Der^p(X,\dist,\meas)$ (resp. $\Der^p_\loc(X,\dist,\meas)$) the space of derivations $\bb$ such that
$\abs\bb\in L^p(X,\meas)$ (resp. $L^p_\loc(X,\meas)$).
\end{definition}

Notice that, thanks to the locality of the minimal relaxed slope $|\rmD f|$ on Borel sets, the choice of $\Lipb(X)$ instead of $\Lipbs(X)$ is not so 
relevant, because any functional $\bb: \Lipbs(X) \to L^0(X,\meas)$ satisfying the locality property can be
canonically extended to a derivation $\tilde\bb$, without increasing its norm, by setting
$$
\tilde{\bb}(f)(x)=\bb(f_R)(x)\qquad\text{$\meas$-a.e.~in $B_R(\bar x)$,}
$$
where $f_R\in\Lipbs(X)$ coincides with $f$ in $B_R(\bar x)$.
As a matter of fact, when integrals are involved, we shall more often work with $\Lipbs(X)$.

As illustrated in details in \cite{Gigli}, arguments analogous to \cite{Ambrosio-Kirchheim}
provide the chain rule (with $\phi:\setR\to\setR$ Lipschitz)
$$
\bb(\phi\circ f)=(\phi'\circ f)\bb(f)\qquad\text{$\meas$-a.e.~in $X$, for every $f\in\Lipb(X)$}
$$
and the Leibniz rule
$$
\bb(fg) = \bb(f)g + f\bb(g) \qquad\text{$\meas$-a.e.~in $X$, for every $f,\,g\in\Lipb(X)$}
$$
as a consequence of additivity and locality (in connection with the Leibniz property, see also \autoref{rem:pre} below).

Any $\bb\in\Der^2(X,\dist,\meas)$ can be canonically extended to $H^{1,2}(X,\dist,\meas)$ if
$\Ch$ is quadratic, retaining the condition $|\bb(f)|\leq |\bb||\rmD f|$ using the existence, for all $f\in H^{1,2}(X,\dist,\meas)$,
of Lipschitz functions $f_n$ with $|\rmD (f-f_n)|\to 0$ in $L^2(X,\meas)$ (so that $\bb(f_n)$ is a Cauchy sequence in $L^1(X,\meas)$).
More generally, this extension is possible whenever $H^{1,2}(X,\dist,\meas)$ is reflexive.

We can now define pre-derivations by restricting the domain to $\Algebra_\bs$ and imposing a weaker locality condition,
with $\Lipa (f)$ in place of $|\rmD f|$ in the right hand side. As for derivations, thanks to locality on open sets of
$\Lipa$, there is essentially no difference in considering
$\Algebra$ or $\Algebra_\bs$ as the domain of a pre-derivation.

\begin{definition}[Pre-derivations]
A \emph{pre-derivation} on a metric measure space $(X,\dist,\meas)$ is a
$\setQ$-linear functional $\bb: \Algebra_\bs\to L^0(X,\meas)$ which, for some $h\in L^0(X,\meas)$, satisfies the locality property
$$\abs{\bb(f)} \le h \Lipa (f) \qquad \text{$\meas$-a.e.~in $X$, for all $f\in\Algebra_\bs$.}$$
\end{definition}

\begin{remark}\label{rem:pre}
Any pre-derivation satisfies the Leibniz rule as a consequence of additivity and locality. 
Indeed, it is sufficient to optimize over rational $\lambda$ and $\mu$ the $\meas$-a.e.~inequality
\[
\begin{split}
\abs{\bb(fg)-\lambda\bb(g)-\mu\bb(f)} &\leq
h \Lipa \bigl(fg-\lambda g-\mu f\bigr) = h \Lipa \bigl((f-\lambda)(g-\mu)\bigr) \\
&\leq h \left[\abs{f-\lambda}\Lipa (g-\mu)+\abs{g-\mu}\Lipa (f-\lambda)\right(].
\end{split}
\]
\end{remark}

\begin{remark} [Different axiomatizations]
Assume that $\bb:\Lipb(X)\to L^0(X,\meas)$ is a linear map satisfying the weaker
locality condition 
$$
\abs{\bb(f)} \le h \Lipa (f)\qquad\text{$\meas$-a.e.~in $X$.} 
$$
for some $h\in L^1_\loc(X,\meas)$ and the continuity w.r.t.~flat convergence
\begin{equation}\label{eq:continuity_of_der}
\text{$f_n\to f$ in $\Lipb(X)$ implies $\bb(f_n)\to\bb(f)$ weakly in $L^1_\loc(X,\meas)$.}
\end{equation}
Then, $\bb\in\Der^1(X,\dist,\meas)$ and $|\bb|\leq h$.
Indeed, the inequality can be improved, getting $|\rmD f|$ in the right hand side, using the continuity property and 
\autoref{tapprox}. Notice that an even weaker locality property will be considered in 
\autoref{lem:stimanorma1}~(ii), and that
\eqref{eq:continuity_of_der} is satisfied when $\div\bb\in L^1_\loc$, see \autoref{lem:continuity_div}.
\end{remark}

\begin{remark} [Derivations as duals of the cotangent bundle] \label{rdual} {\rm 
Notice that $\Der(X,\dist,\meas)$ has a natural structure of $L^\infty(X,\meas)$-module, and that
$\Der^2(X,\dist,\meas)$ has the structure of $L^2(X,\meas)$-normed module, provided by the map
$\bb\mapsto |\bb|$. More specifically, we shall use at some stage that, if $H^{1,2}(X,\dist,\meas)$ is reflexive
(this assumption is not needed in \cite{Gigli} because Gigli's derivations are already defined on
$H^{1,2}(X,\dist,\meas)$), $\Der^2(X,\dist,\meas)$ can be canonically and isometrically
identified with the dual of the cotangent bundle $T^*(X,\dist,\meas)$ introduced in \cite{Gigli}; in this respect, $\Der^2(X,\dist,\meas)$ can
be thought as the collection of all $L^2$ sections of the tangent bundle of $(X,\dist,\meas)$.

Let us recall first the construction of $T^*(X,\dist,\meas)$ from \cite{Gigli}. It arises as the completion of the $L^2(X,\meas)$ normed 
pre-module consisting of finite collections $(E_i,f_i)$ with $E_i\in\Borel(X)$ partition of $X$ and $f_i\in H^{1,2}(X,\dist,\meas)$.
Two collections $(E_i,f_i)$ and $(F_j,g_j)$ are identified if $f_i=g_j$ $\meas$-a.e.~in $E_i\cap F_j$. The sum of $(E_i,f_i)$ and
$(F_j,g_j)$ is defined in the natural way by taking the collection $(E_i\cap F_j,f_i+g_j)$ and the local norm 
$|(E_i,f_i)|\in L^2(X,\meas)$ of $(E_i,f_i)$ is defined by 
$$
|(E_i,f_i)|(x):=|\rmD f_i|(x), \quad \text{$\meas$-a.e.\ $x \in E_i$.}
$$
Thanks to the locality properties of the minimal relaxed slope, this definition does not depend on the choice of
the representative. Multiplication by Borel functions $\chi$ taking finitely many values is defined by
$$
\chi(E_i,f_i)=(E_i\cap F_j,z_jf_i)\qquad\text{with $\chi=\sum_{j=1}^Nz_j\chi_{F_j}$}
$$
and satisfies $|\chi (E_i,F_i)|=|\chi||(E_i,F_i)|$. By completion of this structure we obtain the $L^2(X,\meas)$ normed module
$T^*(X,\dist,\meas)$.

Now, given any $\bb\in\Der^2(X,\dist,\meas)$ we can define
$$
L_\sbb([E_i,f_i]):=\sum_i\chi_{E_i}\bb(f_i)
$$
to obtain a $L^1(X,\meas)$-valued functional on the pre-module which satisfies 
$$
|L_\sbb([E_i,f_i])|\leq |\bb||[E_i,f_i]|\qquad\text{$\meas$-a.e.~in $X$}
$$
and is easily seen to be linear w.r.t.~multiplications by Borel functions $\chi$ with finitely many values. By completion,
$\hat\bb$ extends uniquely to a continuous $L^\infty(X,\meas)$-module homomorphism from $T^*(X,\dist,\meas)$ to $L^1(X,\meas)$
with $|L_\sbb|\leq |\bb|$. From the definition of $|\bb|$ we immediately obtain that $|L_\sbb|=|\bb|$.
Conversely, any continuous $L^\infty(X,\meas)$-module homomorphism $L:T^*(X,\dist,\meas)\to L^1(X,\meas)$
induces $\bb\in \Der^2(X,\dist,\meas)$, via the formula $\bb(f):=L([X,f])$. 
}\end{remark}

In the following lemma, to derive a strict convexity property of the norm of derivations, we make the assumption that $\Ch$ is quadratic, which amounts to say that the cotangent bundle is a Hilbert $L^2(X,\meas)$ module. Of course this convexity could be equally well be taken as an assumption, but we preferred this path (and the use
of \autoref{rdual}) because the quadraticity condition on $\Ch$ is the most established assumption in the recent literature
on $\RCD$ spaces.

\begin{lemma}\label{lem:strict_convexity}
If $\Ch$ is quadratic, for all $\cc,\,\cc'\in\Der(X,\dist,\meas)$ the condition
\begin{equation}\label{eq:parallel}
|\cc+\cc'|=|\cc|+|\cc'|\qquad\text{$\meas$-a.e.~in $X$}
\end{equation}
implies 
$$\cc(f)=\lambda (\cc+\cc')(f)\qquad\text{$\meas$-a.e.~in $X$ for all $f\in\Lipb(X)$,}$$ 
with $\lambda=|\cc|/(|\cc+\cc'|)$ on $\{|\cc+\cc'|>0\}$ (thus independent of $f$).  
\end{lemma}
\begin{proof} Let $E=\{x \in X:\ \min\{|\cc|(x),|\cc'|(x)\}>0\}$. We obviously need only to show that
$\cc(f)=\lambda (\cc+\cc')(f)$ $\meas$-a.e.~in $E$ for all $f\in\Lipb(X)$.
Possibly multiplying $\cc$ and $\cc'$ by a positive Borel function $\chi$ with
$\int_X\chi (|\cc|^2+|\cc'|^2)\di\meas<\infty$ we see that it is not restrictive to assume that both $\cc$ and $\cc'$ belong to
$\Der^2(X,\dist,\meas)$. Now, the assumption that $\Ch$ is quadratic yields that the cotangent bundle $T^*(X,\dist,\meas)$ 
described in \autoref{rdual} is a $L^2(X,\meas)$-normed Hilbert module. By the identification introduced in that remark,
the same holds for $\Der^2(X,\dist,\meas)$. Hence, we can use the bilinear map $\langle \plchldr,\plchldr\rangle_x$ on
$\Der^2(X,\dist,\meas) \times \Der^2(X,\dist,\meas)$ provided by \autoref{tgigli} to write the assumption \eqref{eq:parallel} in the form
$$
\langle\cc,\cc'\rangle_x=|\cc|(x)|\cc'|(x)\qquad\text{$\meas$-a.e.~$x\in X$.}
$$ 
Fix $\bar x\in X$, define now $E_n=\{x\in E:\ \min\{|\cc|(x),|\cc'|(x)\}\geq 1/n\}$ and set
$$
\cc_{n,R}:=\chi_{B_R(\bar x)\cap E_n}\frac{1}{|\cc|}\cc,\qquad \cc_{n,R}':=\chi_{B_R(\bar x)\cap E_n}\frac{1}{|\cc'|}\cc'.
$$
Using $L^\infty(X,\meas)$ bilinearity gives
$$
\langle\cc_{n,R},\cc_{n,R}'\rangle_x=\chi_{B_R(\bar x)\cap E_n}\qquad\text{$\meas$-a.e.~in $X$.}
$$ 
By integration we obtain
$$
\int_X|\cc_{n,R}-\cc_{n,R}'|^2\di\meas=0,
$$
hence $\cc_{n,R}=\cc_{n,R'}$. By approximation with respect to the parameters $n$ and $R$, this gives
$|\cc'|\cc(f)=|\cc|\cc'(f)$ $\meas$-a.e.~on $E$ for all $f\in\Lipb(X)$ and thus, with a simple algebraic manipulation,
the result. 
\end{proof}

\begin{definition}[Divergence of (pre-)derivations]
We say that $\bb\in\Der^1_\loc(X,\dist,\meas)$ has divergence in $L^1_\loc$ if there exists
$g\in L^1_\loc(X,\meas)$ satisfying
\begin{equation}\label{eq:divder}
-\int_X \bb(f) \di\meas = \int_X f g\di\meas
\qquad\text{for all $f\in\Liploc(X)$.}
\end{equation}
The function $g$ is uniquely determined and it will be denoted by $\div\bb$.

The case of a pre-derivation is analogous: the only difference is that $\Liploc(X)$ has to be replaced in \eqref{eq:divder}
by $\Algebra_\bs\subset\Algebra$.
\end{definition}

\begin{lemma}\label{lem:extension-pre-derivation}
Let $\bb$ be a pre-derivation with $h\in L^2_\loc(X,\meas)$ and $\div\bb\in L^1_\loc(X,\meas)$.
Then there exists a unique extension to $\tilde\bb\in\Der^2_\loc(X,\dist,\meas)$, with $|\tilde\bb| \le h$ (and same divergence).
\end{lemma}

\begin{proof} By the locality property, it is sufficient to build $\bb$ on $\Lip_\bs(X)$.
Given $f \in \Lip_\bs(X)$, let $f_n \in \Algebra_\bs$ be an approximating sequence as provided by \autoref{tapprox_refined}, 
i.e.~such that $\Lipa (f_n) \to \abs{\rmD f}$ in $L^2(X,\meas)$. Actually, up to truncating $f_n$, we may assume that $\abs{f_n} \le \norm{f}_\infty$, for $n\ge 1$. 
By the bound $|\bb(f_n)|\le h\Lipa (f_n)$ and \autoref{lem:DPett}~(i),  
the family $\bb(f_n)$ has weak limit points in $L^1_\loc(X,\meas)$ and any weak limit point, 
denoted $\tilde\bb(f)$, satisfies
\[  |\tilde\bb(f)| \le h|\rmD f|. \]
We only have to show that such limit is uniquely determined by the formula
$$
\int_X \phi \tilde\bb(f) \di\meas = - \int_X f \phi \div \bb \di\meas - \int_X f\bb(\phi) \di\meas\qquad\forall\phi\in\Algebra_\bs,
$$
so that the whole family $\bb(f_n)$ will be weakly convergent and this
will give that $\tilde\bb$ extends $\bb$ and that the map $f\mapsto\tilde\bb(f)$ is linear. Indeed, by definition of divergence, for any $\phi \in \Algebra_\bs$
one has $f_n\phi\in\Algebra_\bs$ and thus
\[
\int_X \phi \bb(f_n) \di\meas = - \int_X f_n \phi \div \bb \di\meas - \int_X f_n \bb(\phi) \di\meas.
\]
By convergence of both sides, the same identity holds with $f$ in place of $f_n$ and $\tilde{\bb}(f)$ in place of $\bb(f_n)$.
\end{proof}

\begin{lemma}[Continuity criterion w.r.t.~flat convergence]\label{lem:continuity_div} Let $\bb\in\Der^1_\loc(X,\dist,\meas)$ and assume that
either $\bb\in\Div^1_\loc(X,\dist,\meas)$, or $H^{1,2}(X,\dist,\meas)$ is reflexive.
Then \eqref{eq:continuity_of_der} holds.
\end{lemma}
\begin{proof} Let $f_n\to f$ in $\Lipb(X)$. If $\bb$ has divergence in $L^1_\loc$, we can use the formula
$$
\int_X \bb(f)\varphi\di\meas=-\int_X \bb(\varphi) f\di\meas-\int_X f\varphi\div\bb\di\meas\qquad\varphi\in\Lip_\bs(X),
$$
derived from the Leibniz rule to show the continuity property \eqref{eq:continuity_of_der}. In case
of a reflexive Sobolev space $H^{1,2}(X,\dist,\meas)$,  assuming (by locality) that the supports of $f_n$ are equibounded,
we obtain that $f_n\to f$ weakly in $H^{1,2}(X,\dist,\meas)$; hence, if $w$ is any weak $L^1$ limit point of
$\bb(f_n)$, we can obtain $w$ as the limit of $\bb(g_n)$, where $g_n$ are finite convex combinations of $f_n$
converging to $f$ strongly in $H^{1,2}(X,\dist,\meas)$. If follows that $\bb(g_n)\to\bb(f)$, hence $w=\bb(f)$.
\end{proof}

In the sequel we shall use the notation 
$$
\Div^q(X,\dist,\meas):=\left\{\bb\in \Der^1_\loc(X,\dist,\meas):\ \div\bb\in L^q(X,\meas)\right\}
$$
and the analogous notation $\Div^q_\loc(X,\dist,\meas)$.

\part{Weak/strong convergence of derivations}\label{pt:derivations}

Given $(\meas_n)\subset\Meas^+_\loc(X)$ weakly convergent to $\meas\in\Meas^+_\loc(X)$ and $\bb_n\in\Der^1_\loc(X,\dist,\meas_n)$,
we consider notions of local convergence of $\bb_n$ to $\bb\in\Der^1_\loc(X,\dist,\meas)$, two weak ones and a strong one. We study these two
concepts in the next two sections.

\section{Weak convergence of derivations and compactness}\label{sec:two_weak}

We say that $\bb_n\in\Der^1_\loc(X,\dist,\meas_n)$ weakly converge to $\bb\in\Der^1_\loc(X,\dist,\meas)$ in duality with $\Algebra_\bs$, and
write $\bb_n\Weakto{\Algebra_\bs}\bb$ if, for all $f\in\Algebra_\bs$, 
$\bb_n(f)\meas_n$ weakly converge to $\bb(f)\meas$ as measures, i.e.
$$
\lim_{n\to\infty}\int_X \bb_n(f) v\di\meas_n=\int_X \bb(f) v\di\meas\qquad\forall v\in \Cb(X).
$$

\begin{theorem}[Weak compactness of derivations]\label{thm:compactness}
Let $\bb_n\in\Der^2_\loc\cap\Div^1_\loc(X,\dist,\meas_n)$ satisfying for some $\bar x\in X$ and all $R>0$ the conditions 
\[
\sup_n \int_{B_R(\bar x)} |\bb_n|^2 \di\meas_n < \infty, \qquad
\sup_n \int_{B_R(\bar x)} \Theta(|\div\bb_n|)\di\meas_n < \infty
\]
for some function $\Theta=\Theta_R:[0,\infty)\to [0,\infty]$ with more than linear growth at infinity.
Then there exists a subsequence $\bb_{n(k)}$ weakly convergent to $\bb\in\Der^2_\loc\cap \Div^1_\loc(X,\dist,\meas)$.
\end{theorem}

\begin{proof}
Let $\phi \in \Algebra_\bs$. Since $|\bb_n(\phi)|$ is bounded in $L^2$, up to a subsequence there exists a weak $L^2$ limit, that we denote by $\bb(\phi)$. By a diagonal argument, we may assume that $\bb_n(\phi)$ weakly converge in $L^2$ to $\bb(\phi)$ for every $\phi \in\Algebra_\bs$. 
We show that $\phi \mapsto \bb(\phi)$ defines a pre-derivation with $\div \bb \in L^1_\loc(X,\meas)$, hence by \autoref{lem:extension-pre-derivation} we conclude.

Since $\setQ$-linearity is preserved by weak limits we only have to notice that locality of $\bb$ follows from
\[ |\bb_n(\phi) | \le \abs{\bb_n} \abs{\rmD_n \phi} \le \abs{\bb_n} \Lipa (\phi),\quad \text{$\meas_n$-a.e.~on $X$,}\]
where $\phi\in\Algebra_\bs$ and $|\rmD_n \phi|$ denotes the minimal relaxed slope of $\phi$ with respect to the measure $\meas_n$. Hence, by weak convergence in $L^2_\loc$ of $|\bb_n|$  to some $h \in L^2_\loc(X,\meas)$ (possibly taking a further subsequence) and upper semicontinuity of $\Lipa (\phi)$, we conclude that
\[ \abs{\bb(\phi) } \le h \Lipa(\phi),\quad \text{$\meas$-a.e.~on $X$.}\]
Finally, a similar limiting argument proves that $\div\bb_n$ has limit points w.r.t.~the weak $L^1_{\rm loc}$ 
convergence, and that limit points correspond to $\div\bb$.\qedhere
\end{proof}

Our second notion of weak convergence depends somehow on structural assumptions on the metric measure structures, and actually
only on the limit one. Let us assume that the limit metric measure structure $(X,\dist,\meas)$ satisfies the following two regularity
assumptions  (recall that $|\rmD f|$ stands for the minimal relaxed slope, see \autoref{appendix:slopes}):

\begin{itemize}
\item[(a)] $\Ch$ is quadratic; 
\item[(b)] there exists  $\omega:(0,\infty)\to [0,\infty)$ such that $\omega(0_+)=1$ and
\begin{equation}\label{eq:basic1}
\Lipa (P_tf)\leq \omega(t)\sqrt{P_t |\rmD f|^2}\quad\text{pointwise in $\supp\meas$}
\end{equation}
for all $f\in\Lipb(X)\cap H^{1,2}(X,\dist,\meas)$ and $t>0$.
\end{itemize}

These assumptions play a role in Lemma~\ref{lem:stimanorma2} and, in a more essential way, in the proof
of Lemma~\ref{lem:stimanorma1}(ii).
As illustrated in \autoref{appendix:slopes}, all these assumptions hold in $\RCD(K,\infty)$ metric measure spaces, in particular
\eqref{eq:basic1} holds with with $\omega(t)=e^{-Kt}$.  Notice that, since $P_tf$ is only defined up to $\meas$-negligible
sets, using McShane's extension theorem we can and will improve (b) by asking the validity of \eqref{eq:basic1} on the whole
of $X$. 

In view of  \eqref{eq:basic1} it is natural to consider the countable class
\begin{equation}\label{eq:def_Cinfty}
P_{\setQ_+}\Algebra_\bs:=\left\{P_tf:\ f\in\Algebra_\bs,\,\,t\in\setQ_+\right\}\subset\Lipb(X).
\end{equation}
Notice that, thanks to \autoref{tapprox_refined}, $P_{\setQ_+}\Algebra_\bs$ is dense in $H^{1,2}(X,\dist,\meas)$.
 Strictly speaking $P_{\setQ_+}\Algebra_\bs$ should be thought as a subset of $\Lipb(\supp\meas)$, but we understand that
any function in this class is extended to the whole of $X$, with the same supremum and Lipschitz constant, and we think to this
collection as a subset of $\Lipb(X)$.

\begin{lemma}\label{lem:stimanorma2}
Under assumption (a), any $\bb\in\Der(X,\dist,\meas)$ is uniquely determined by its values on $P_{\setQ_+}\Algebra_\bs$.
\end{lemma}
\begin{proof}
We need to prove that $\bb\equiv 0$ on $P_{\setQ_+}\Algebra_\bs$ implies $\bb=0$.
It is not restrictive to assume $\bb\in\Der^2(X,\dist,\meas)$. Since, thanks to assumption (a), 
$P_sf\to f$ in $H^{1,2}(X,\dist,\meas)$ as $s\downarrow 0$,
we obtain that $\bb\equiv 0$ on $\Algebra_\bs$. Then, \autoref{tapprox_refined} can be applied.
\end{proof}

Motivated by \autoref{lem:stimanorma2}, we can define another notion of weak convergence.

\begin{definition}
We say that $\bb_n\in\Der^1(X,\dist,\meas_n)$ weakly converge to $\bb\in\Der^1(X,\dist,\meas)$ in duality with $P_{\setQ_+}\Algebra_\bs$, and
write $\bb_n\Weakto{P\Algebra_\bs}\bb$ if, for all $f\in P_{\setQ_+}\Algebra_\bs$, 
$\bb_n(f)\meas_n$ weakly converge to $\bb(f)\meas$ as measures, i.e.
\begin{equation}\label{eq:be_careful}
\lim_{n\to\infty}\int_X \bb_n(f) v\di\meas_n=\int_X \bb(f) v\di\meas\qquad\forall v\in \Cbs(X).
\end{equation}
\end{definition}

Notice that in \eqref{eq:be_careful} the integrals on the left hand side \textit{depend} on the extension of have chosen
of $f$ from $\supp\meas$ to $X$; nevertheless, the compactness result presented below shows that still \eqref{eq:be_careful} makes sense.
With this new notion of weak convergence we can remove the bounds on divergence and, at the same time, consider
more general growth conditions (not necessarily quadratic) on $\bb_n$.

\begin{theorem}\label{thm:compa2}
Assume that $(X,\dist,\meas)$ satisfies the regularity assumptions (a), (b) and
let $\bb_n\in\Der(X,\dist,\meas_n)$ be such that 
\[
\liminf_{n\to\infty}\int_X\Theta(|\bb_n|)\di\meas_n<\infty
\]
for some function $\Theta:[0,\infty)\to [0,\infty]$ with more than linear growth at infinity.
Then there exist a subsequence $\bb_{n(k)}$ and $\bb\in\Der(X,\dist,\meas)$ with
$\bb_n\Weakto{P\Algebra_\bs}\bb$. In addition, if $\Theta$ is convex and lower semicontinuous, one has
$$
\int_X\Theta(|\bb|)\di\meas\leq\liminf_{k\to\infty}\int_X\Theta(|\bb_{n(k)}|)\di\meas_{n(k)}<\infty.
$$
\end{theorem}
\begin{proof}
Let $a \in P_{\setQ_+}\Algebra_\bs$. Since $|\bb_n(a)|$ satisfies the assumptions of \autoref{lem:DPett}~(ii), 
up to a subsequence there exists a weak $L^1$ limit, that we denote by $\bb(\phi)$. By a diagonal argument, we may assume that $\bb_n(a)$ weakly $L^1_\loc$ 
converge to $\bb(a)$ for every $a \in P_{\setQ_+}\Algebra_\bs$. For the same reason is also not restrictive to assume that $|\bb_n|$ weakly converge in $L^1_\loc$ to
$h$.

Since $\setQ$-linearity is preserved by weak local limits, taking \autoref{lem:stimanorma1}~(ii) into account it is sufficient to prove
 $$
 |\bb(a)|\leq h\Lipa (a)\qquad\text{$\meas$-a.e.~in $X$, for all $a\in P_{\setQ_+}\Algebra_\bs$.}
$$
This can be achieved as in the proof of \autoref{thm:compactness} passing to the limit in the inequality
$|\bb_n(a)|\leq |\bb_n|\Lipa (a)$ and using the upper semicontinuity of the asymptotic Lipschitz constant.\qedhere
\end{proof}

\begin{remark} [Comparison of $\bb_n\Weakto{\Algebra_\bs}\bb$ and $\bb_n\Weakto{P\Algebra_\bs}\bb$]\label{rem:compaweak}
Assume that we are in the setting of both \autoref{thm:compactness} and \autoref{thm:compa2}, namely with uniform local equi-integrability bounds on the divergences of $\bb_n$
and assumptions (a), (b) on the limit metric measure structure. Under this assumption, since both notions of convergence are sequentially compact,
we immediately get from \autoref{lem:continuity_div} that $\bb_n\Weakto{\Algebra_\bs}\bb$ if and only if $\bb_n\Weakto{P\Algebra_\bs}\bb$.
\end{remark}

\section{Strong convergence of derivations}

Still under the assumptions (a), (b) of the previous section on the metric measure structure, we want to study conditions under which
weak convergence in duality with $P_{\setQ^+}\Algebra_\bs$ improves to strong convergence, defined as
follows. 

\begin{definition}[Strong convergence of derivations]
We say that $\bb_n\in\Der^1_\loc(X,\dist,\meas_n)$ strongly converge to $\bb\in\Der^1_\loc(X,\dist,\meas)$ if,
for all $f\in P_{\setQ^+}\Algebra_\bs$, $\bb_n(f)$ converge in measure to $\bb(f)$. For $p\in (1,\infty)$,
if for all $f\in P_{\setQ_+}\Algebra_\bs$ the stronger condition 
\begin{equation*}
\limsup_{n\to\infty}\int_{B_R(\bar x)}|\bb_n(f)|^p\di\meas_n\leq\int_{B_R(\bar x)}|\bb(f)|^p\di\meas<\infty
\end{equation*}
holds for an unbounded family of $R$'s,  we say that $\bb_n$ converge strongly in $L^p_\loc$ to $\bb$. Analogously,
if
$$
\limsup_{n\to\infty}\int_X|\bb_n(f)|^p\di\meas_n\leq\int_X|\bb(f)|^p\di\meas<\infty
$$
we say that $\bb_n\to\bb$ strongly in $L^p$.

\end{definition}

The terminology is a bit misleading, because strong convergence of $\bb_n$ should 
be understood as a ``componentwise'' strong convergence, where the components
of the ``vector field'' are given by $P_{\setQ_+}\Algebra_\bs$. As discussed in \autoref{rem:sono_simili}, 
under uniform $L^p_\loc$ bounds on $\bb_n(f)$ for some $p>1$, the convergence in measure
of $\bb_n(f)$ to $\bb(f)$ yields convergence in $L^q_\loc$ for all $q\in (1,p)$ and therefore 
$\bb_n\Weakto{P\Algebra_\bs}\bb$. 


The main goal of this section is to provide a sufficient condition for strong convergence. To this aim, it is
convenient to extend $\Algebra_\bs$ to $\Algebra_*$, defined as follows:
$\Algebra_*$ is the smallest algebra, lattice and vector space over $\setQ$ containing $\Algebra$ and
the set $P_{\setQ_+}\Algebra_\bs$ in \eqref{eq:def_Cinfty}.
The set $\Algebra_*$ is still countable and we denote by $a$ the generic element of $\Algebra_*$.

In the following lemma we show how to to build a derivation from a functional $\cc:\Algebra_*\to L^0(X,\meas)$ satisfying
a (very) weak locality property and a growth bound with the global Lipschitz constant, see \eqref{eq:very_weak}.

\begin{lemma}\label{lem:stimanorma1}
Assume that $(X,\dist,\meas)$ satisfies the regularity assumptions (a), (b).
\begin{itemize}
\item[(i)] Let $\cc:\Algebra_*\to L^0(X,\meas)$ be a linear functional satisfying the locality property $\cc(a)=0$ $\meas$-a.e.
in any open set $U$ where $a$ is constant, and 
\begin{equation}\label{eq:very_weak}
|\cc(a)|\leq h\Lip(a)\qquad\text{$\meas$-a.e.~in $X$, for all $a\in\Algebra_*$}
\end{equation}
for some $h\in L^0(X,\meas)$. Then $\cc$ is a pre-derivation, more precisely, \eqref{eq:very_weak} holds with
$\Lipa (a)$ in the right hand side.
\item[(ii)] If $\cc:P_{\setQ_+}\Algebra_\bs\to L^0(X,\meas)$ is $\setQ$-linear and satisfies
\begin{equation}\label{eq:very_weak_bis}
 |\cc(a)|\leq h\Lipa (a)\qquad\text{$\meas$-a.e.~in $X$, for all $a\in P_{\setQ_+}\Algebra_\bs$,}
\end{equation}
for some $h\in L^0(X,\meas)$, then there exists a derivation $\bb$ satisfying $|\bb|\leq h$ which extends $\cc$. 
\end{itemize}
\end{lemma}
\begin{proof}(i) Let $D$ be as in \eqref{eq:setD}. We want to improve the bound on $|\cc(a)|$ to $|\cc(a)|\leq h\Lipa (a)$ $\meas$-a.e.~in $X$ for
all $a\in\Algebra_*$. To this aim, we fix $a\in\Algebra_*$ and a bounded open set $U\subset X$ and define
$L_U$ to be the Lipschitz constant of the restriction of $a$ to $U$. For $x_0\in U$, choose $x_0'\in D$ and $r>0$ sufficiently
small, such that $B_{r/2}(x_0')\subset B_r(x_0)\subset U$, $rL_U/2\in\setQ$. Then, for
$k\in\setQ^+$ larger than $r/2$ and $a_0\in\setQ$, $a_0\neq a(x_0)$, we consider the function $A=\max\{A_1,A_2\}\in\Algebra_*$, with
$$
A_1(x):=\min\{|a(x)-a_0|,\frac {r}2 L_U\},\qquad A_2(x):=L_U\min\{\dist(x,x_0'),k\}.
$$
Since ${\rm Lip}(A_1\vert_U)\leq L_U$, ${\rm Lip}(A_2)\leq L_U$ and the set $\{A=A_1\}$
is contained in $B_{r/2}(x_0')$, which is contained in $B_r(x_0)$ and in $U$, 
we obtain that $\Lip(A)\leq L_U$. Choosing $a_0$ sufficiently close to $a(x_0)$ we obtain that
$A_1=|a-a_0|$ in a neighbourhood $V$ of $x_0$, and choosing $x_0'$ with
$$
L_Ud(x_0,x_0')<|a(x_0)-a_0|
$$
we obtain that $A=|a-a_0|$ in a neighbourhood $V$ of $x_0$.
Hence, locality in $V\cap\{a>a_0\}$ and in $V\cap \{a<a_0\}$ give
$$
|\cc(a)|=|\cc(A)|\leq h\Lip(A)\leq h L_U\qquad\text{$\meas$-a.e.~in $V\setminus\{a=a_0\}$.}
$$ 
With two different choices of $a_0$, this proves that for any $x_0\in U$ there exists a neighbourhood $V$ of $x_0$ 
such that $|\cc(a)|\leq hL_U$ $\meas$-a.e.~in $V$.
Since $x_0\in U$ is arbitrary, and since $\meas$ is concentrated on a $\sigma$-compact set,
it follows that $|\cc(a)|\leq h L_U$ $\meas$-a.e.~in $U$. If $\{U_i\}$ is a countable basis for the open sets
of $(X,\dist)$ we can find a $\meas$-negligible set $N$ such that
$$
|\cc(a)|(x)\leq h(x)L_{U_i}\qquad\forall x\in U_i\setminus N,\,\,\forall i.
$$
Since $\Lipa a(x)=\inf_{U_i\ni x}L_{U_i}$ for all $x\in X$, we conclude.

(ii) First we improve the upper bound \eqref{eq:very_weak_bis} to 
\begin{equation}\label{eq:Young6}
|\cc(a)|\leq h|\rmD a|\qquad\text{$\meas$-a.e.~in $X$, for all $a\in P_{\setQ^+}\Algebra_\bs$.}
\end{equation} 
Indeed, thanks to assumption (b), one has
$$
|\cc(P_s a)|\leq \omega(s) h \sqrt{P_s|\rmD a|^2}
\quad\text{$\meas$-a.e.~in $X$, for all $s\in\setQ^+$, $a\in\Algebra_\bs$}
$$
and, writing $a=P_ta'$, we can use the linearity of $P_t$ to get
$$
|\cc(P_s a)-\cc(a)|=|\cc(P_t(P_sa'-a'))|\leq\omega(t) h \sqrt{ P_t \bigl(|\rmD (P_s a'-a')|^2\bigr)}.
$$
Then, taking the limit as $\setQ^+\ni s\to 0^+$ of the inequality
provides the result. Then, we build $\bb$ on $\Lip_\bs(X)$ by approximation, as a simple consequence of \eqref{eq:Young6} and the density
of $P_{\setQ^+}\Algebra_\bs$ in $H^{1,2}(X,\dist,\meas)$, arguing as in \autoref{lem:extension-pre-derivation}: if
$f\in\Lip_\bs(X)$ and $a_n\in P_{\setQ^+}\Algebra_\bs$ satisfy $|\rmD (a_n-f)|\to 0$ in $L^2(X,\meas)$ and $\meas$-a.e., 
thanks to \eqref{eq:Young6} the sequence $\bb(a_n)$ converges $\meas$-a.e.~and we define $\bb(f)$ as its limit. This
provides a derivation $\bb$ with $|\bb|\leq h$ and $\bb\equiv\cc$ on $P_{\setQ^+}\Algebra_\bs$.
\end{proof}

In the following theorem we show how to improve weak convergence, in duality with $P_{\setQ_+}\Algebra_\bs$,
to strong convergence. In connection with the theory of flows, we shall also consider time-dependent derivations and
therefore a time averaged version of weak convergence, deriving as a consequence strong convergence in measure,
with $\leb^1\times\meas_n$ and $\leb^1\times\meas$ as reference measures. 
Of course in the simpler autonomous case we get the improvement from weak to strong convergence.

\begin{theorem} \label{thm:strong-stability-hilbert} Assume that $(X,\dist,\meas)$ satisfies the regularity assumptions (a), (b), 
that $\bb_{n,t}\in\Der^1_\loc(X,\dist,\meas_n)$, $\bb_t\in\Der^1_\loc(X,\dist,\meas)$ satisfy
$$
\lim_{n\to\infty}\int_0^T\chi(t)\int_X\bb_{n,t}(a)v\di\meas_n\di t=
\int_0^T\chi(t)\int_X\bb_t(a)v\di\meas\di t
$$
for all $\chi\in C_c(0,T),\,\,v\in\Cbs(X),\,\,a\in P_{\setQ_+}\Algebra_\bs$, and that
\begin{equation}
\label{eq:lim-sup-theta-bis}
\limsup_{n\to\infty}\int_0^T\int_X\Theta(|\bb_{n,t}|)\di\meas_n\,\di t\leq
\int_0^T\int_X\Theta(|\bb|_t)\di\meas\,\di t<\infty
\end{equation}
with $\Theta:[0,\infty)\to [0,\infty)$ strictly convex and having more than linear growth at infinity.
Then $\bb_{n,t}$ converge strongly to $\bb_t$ and, if $\Theta(z)=|z|^p$ for some $p>1$, $\bb_n$ converge
to $\bb$ strongly in $L^p$.
\end{theorem}
\begin{proof} {\it Step 1: the dual $\Algebra_*'$ of $\Algebra_*$.} The $\setQ$-linear 
functionals from $\Algebra_*$ to $\setR$
can be considered, in some sense, as ``pointwise'' tangent vectors. More 
precisely, we denote by $\Algebra_*'$ the class of $\setQ$-linear 
functionals
$L:\Algebra_*\to\setR$ satisfying $|L(a)|\leq C\Lip(a)$ for all 
$a\in\Algebra_*$, for some $C\in [0,\infty)$. The smallest $C$ will then
be denoted by $\|L\|$. We endow $\Algebra_*'$ with the coarser topology 
that makes all maps
$L\mapsto L(a)$, $a\in\Algebra_*$, continuous, so that $L\mapsto\|L\|$ 
is lower semicontinuous in $\Algebra_*'$ and the
sets $\{L\in\Algebra_*':\ \|L\|\leq c\}$ are compact for all $c\geq 0$.
Since $\Algebra$ is countable we can easily find a distance in 
$\Algebra_*'$
which induces this topology (on bounded sets).

\noindent
{\it Step 2: limit of $\bb_{n,t}$ in the sense of Young.}  Let us consider the maps $\Sigma_n:(0,T)\times X\to\Algebra_*'$ induced by $\bb_{n,t}$, namely
$$
\Sigma_n(t,x)(a):=\bb_{n,t}(a)(x)
$$
and the push forward measures $$\sigma_n:=(\Id\times\Sigma_n)_\#\leb^1\times\meas_n\in\Meas^+((0,T)\times X\times\Algebra_*').$$ Notice that,
since $\Algebra_*$ is countable, $\sigma_n$ is independent of the choice of Borel representatives of $\bb_{n,t}(a)$,
$a\in\Algebra_*$, and that $\|\Sigma_n(t,x)\|\leq |\bb|_{n,t}(x)$ for $\leb^1\times\meas_n$-a.e.~$(t,x)$. 
Since $\leb^1\times\meas_n$ weakly converge to $\leb^1\times\meas$, they are tight on bounded subsets, hence thanks to
Prokhorov theorem and 
the uniform bound on $\int\Theta(\|L\|)\di\sigma_n$, we can assume with no loss of generality that $\sigma_n$ weakly converge to $\sigma$ in
$(0,T)\times X\times\Algebra_*'$. Using test functions of the form $\psi(t,x)$ we can represent $\sigma=\sigma_{t,x}\otimes(\leb^1\times\meas)$,
i.e.
$$
\int_{(0,T)\times X\times\Algebra_*'}\phi(t,x,a)\di\sigma=\int_0^T\int_X\int_{\Algebra_*'}\phi(t,x,a)\di\sigma_{t,x}(a)\,\di x\,\di t.
$$
From now on, for simplicity of notation, we shall omit the integration domain $(0,T)\times X\times\Algebra_*'$. 

Let $a\in\Algebra_*$ be constant in an open set $A\subset X$.
Using test functions of the form $(t,x,L)\mapsto\chi(t)v(x)|L(a)|$ with $v\in \Cbs(X)$ null on $X\setminus A$
we see immediately that $\sigma_{t,x}$ satisfy the following locality property: 
for $\leb^1\times\meas$-a.e.~$(t,x)\in (0,T)\times A$ one has $L(a)=0$ $\sigma_{t,x}$-a.e.~in $\Algebra_*'$.
Thanks to this property, defining
\begin{equation}\label{eq:march1}
\tilde\bb_t(a)(x):=\int_{\Algebra_*'} L(a)\di\sigma_{t,x}(L)
\end{equation}
using the inequality $|\tilde\bb_t(a)(x)|\leq\Lip(a)\int_{\Algebra_*}\|L\|\di\sigma_{t,x}(L)$ and invoking
\autoref{lem:stimanorma1}~(ii) we obtain
\begin{equation}\label{eq:march11}
|\tilde\bb_t(a)(x)|\leq \Lipa(a)\int_{\Algebra_*}\|L\|\di\sigma_{t,x}(L)\qquad\text{$\meas$-a.e.~in $X$, for all $a\in\Algebra_*$}
\end{equation}
for $\leb^1$-a.e.~$t\in (0,T)$.

Set now $\Sigma(t,x)(a):=\bb_t(a)(x)$, $a\in\Algebra_*$.
Our goal is to show that $\sigma_{t,x}$ is a Dirac mass at $\Sigma(t,x)$ for $\leb^1\times\meas$-a.e.~$(t,x)$. To this aim, we first observe the following
characterization of Dirac masses in $\Algebra_*'$: $\nu\in\Prob(\Algebra_*')$ concentrated on $\{L:\ \|L\|=c\}$ for some $c\geq 0$
is a Dirac mass if and only if, for all $a\in\Algebra_*$,  $\nu$ is supported in one of the ``halfspaces'' $\{L:\ L(a)\geq 0\}$, $\{L:\ L(a)\leq 0\}$. 
Since $\Algebra_*$ is countable, this characterization follows by the implication
$$
L(a)L'(a)\geq 0\quad\forall a\in\Algebra_*\quad\Longrightarrow\quad L=\lambda L'\quad\text{for some $\lambda\geq 0$.}
$$
In turn, the implication above follows by this elementary argument: if $a\in\Algebra_*$ with $L'(a)\neq 0$ and $\lambda=L(a)/L'(a)$, for all
$a'\in\Algebra_*$ the functions $s\mapsto L(a+sa')$, $s\mapsto \lambda L'(a+sa')$ coincide at $s=0$ 
and have the same sign if and only they coincide, i.e.~$L(a')=\lambda L'(a')$.

Now, notice that \eqref{eq:lim-sup-theta-bis} gives
\begin{equation}\label{Young2}
\int \Theta(\|L\|)\di\sigma\leq\liminf_{n\to\infty}\int \Theta(\|L\|)\di\sigma_n\leq\int_X\Theta(|\bb|)\di\meas.
\end{equation}
On the other hand, using test functions $(t,x,L)\mapsto \chi(t)v(x) L(a)$ with $\chi\in C^\infty_c(0,T)$,
$v\in \Cbs(X)$ and $a\in P_{\setQ_+}\Algebra_\bs$ (notice that here we use the more than linear growth of 
$\Theta$) and the convergence of $\bb_{n,t}$ to $\bb_t$, passing to the limit as $n\to\infty$ in the identity
$$
\int \chi(t)v(x)L(a)\di\sigma_n(t,x,L)=\int_0^T\chi(t)\int_X \bb_{n,t}(a)(x)v(x)\di\meas_n(x)\,\di t
$$
and using the arbitrariness of $\chi$ and $v$ we obtain (with $\tilde\bb_t$ as in \eqref{eq:march1})
\begin{equation}\label{eq:Young3}
\tilde{\bb}_t(a)(x)=\bb_t(a)(x)\qquad\forall a\in P_{\setQ_+}\Algebra_\bs
\end{equation}
for $\leb^1\times\meas$-a.e.~$(t,x)\in (0,T)\times X$. From \eqref{eq:march11} it follows that
$$
|\bb_t(a)|(x)\leq \Lipa(a)\int_{\Algebra_*'} \|L\|\di\sigma_x(L)\qquad\forall a\in P_{\setQ_+}\Algebra_\bs
$$
for $\leb^1\times\meas$-a.e.~$(t,x)\in (0,T)\times X$, so that \autoref{lem:stimanorma1}~(ii) gives
$$
|\bb_t|(x)\leq \int_{\Algebra_*'} \|L\|\di\sigma_{x,t}(L)\qquad\text{for $\leb^1\times\meas$-a.e.~$(t,x)\in (0,T)\times X$.}
$$
Combining this information with \eqref{Young2} and the strict convexity of $\Theta$ we obtain that for $\leb^1\times\meas$-a.e.~$(t,x)$ one has
\begin{equation*}
\|L\|=|\bb_t|(x)\qquad\text{for $\sigma_{t,x}$-a.e.~$L\in\Algebra_*'$.}
\end{equation*}

Now, let $E$ be a closed set of $\Algebra_*'$ and define $\cc^i_t:\Algebra_*\to L^0(X,\meas)$ by
$$
\cc^1_t(a)(x):=\int_{E}L(a)\di\sigma_{t,x}(L),\qquad
\cc^2_t(a)(x):=\int_{\Algebra_*'\setminus E}L(a)\di\sigma_{t,x}(L),
$$
so that $\cc_t:=\cc_t^1+\cc_t^2$ coincides with $\bb_t$ on $\Algebra_*$. Notice also that $\cc^i_t$ satisfy the weak locality property 
(thanks to the locality of $\sigma_{t,x}$) and \eqref{eq:very_weak}
with $h^1_t(x)=\int_E\|L\|\di\sigma_{t,x}(X)$ and $h^2_t(x)=\int_{\Algebra_*'\setminus E}\|L\|\di\sigma_{t,x}(X)$. 
From \autoref{lem:stimanorma1}~(i,ii) we obtain that
$\cc^i_t$ induce derivations $\bb^i_t$ which coincide with $\cc^i_t$ on $P_{\setQ^+}\Algebra_\bs$ and satisfy
$$
|\bb^1_t|(x)\leq \int_{E}\|L\|\di\sigma_{t,x}(L),\,\,\,
|\bb^2_t|(x)\leq \int_{\Algebra_*'\setminus E}\|L\|\di\sigma_{t,x}(L)
\quad\text{$\leb^1\times\meas$-a.e.~in $(0,T)\times X$.}
$$
Since, thanks to \autoref{lem:stimanorma1}~(ii), $\bb_t$ is uniquely determined by its values on $P_{\setQ^+}\Algebra_*$, it follows 
that $\bb_t=\bb_t^1+\bb_t^2$ and that $|\bb^1_t|+|\bb^2_t|\leq |\bb_t|$ $\leb^1\times\meas$-a.e.~in $(0,T)\times X$. Therefore we can invoke the
strict convexity \autoref{lem:strict_convexity} to obtain  $\bb^1_{t,x}(a)\bb^2_{t,x}(a)\geq 0$ $\leb^1\times\meas$-a.e.~for all $a\in\Algebra_*$; by letting
$E$ vary in a countable family of closed sets generating the $\sigma$-algebra of $\Algebra_*'$, it
follows that $\leb^1\times\meas$-a.e.~measure $\sigma_{t,x}$ is concentrated on $\{L:\ \|L\|=|\bb_{t,x}|\}$ and
satisfies the above mentioned criterion for being a Dirac mass. Coming back to \eqref{eq:march1} and \eqref{eq:Young3} 
we obtain that $\sigma_{t,x}=\delta_{\Sigma(t,x)}$.

\noindent
{\it Step 3.} We proved that $(\Id\times\Sigma_n)_\#\leb^1\times\meas_n$ weakly converge
to $(\Id\times\Sigma)_\#\leb^1\times\meas$. Recalling the definitions of $\Sigma_n$ and $\Sigma$, if we consider the
continuous map $(t,x,L)\mapsto (t,x,L(a))$ for $a\in P_{\setQ^+}\Algebra_\bs$ fixed, from 
\autoref{prop:passa_al_limite} we obtain that strong convergence
of $\bb_{n,t}(a)$ to $\bb_t(a)$. In addition, if $\Theta(z)=|z|^p$ for some $p>1$, from the second part of 
\autoref{prop:passa_al_limite} we get the strong $L^p$ convergence of $\bb_{n,t}(a)$ to $\bb_t(a)$.
\end{proof}

\section{Convergence of gradient derivations under Mosco convergence}

In this section we use the typical notation of $\Gamma$-calculus, namely
$$
\Gamma(f,g):=\lim_{\epsilon\downarrow 0}\frac{|\rmD (f+\epsilon g)|^2-|\rmD f|^2}{2\epsilon}
\qquad\quad f,\,g\in H^{1,2}(X,\dist,\meas).
$$
Under the assumption that $\Ch$ is quadratic, this is a symmetric and $L^1(X,\meas)$-valued bilinear form,
with $\Gamma(f,f)=|\rmD f|^2$ $\meas$-a.e.~in $X$. This fact, proved first in \cite{AmbrosioGigliSavare14}, can now be seen as
a particular case of \autoref{tgigli}. We can canonically associate to any $f\in H^{1,2}(X,\dist,\meas)$ a 
\emph{gradient} derivation $\bb_f=\bb_{f,\meas}\in\Der^2(X,\dist,\meas)$, defined by
$$
\bb_f(g):=\Gamma(f,g)\qquad\quad g\in H^{1,2}(X,\dist,\meas)
$$
(we already consider, as for all derivations in $L^2$, the extended domain $H^{1,2}(X,\dist,\meas)$).

\begin{lemma}\label{lem:Hodge}
Assume that $\Ch$ is quadratic, let $f\in H^{1,2}(X,\dist,\meas)$, $\bb\in\Der^2(X,\dist,\meas)$ and assume that
$\int_X |\bb|^2\di\meas\leq\int_X|\rmD f|^2\di\meas$
and that $\int_X(\bb(g)-\bb_f(g))\di\meas=0$ for all $g\in H^{1,2}(X,\dist,\meas)$.
Then $\bb=\bb_f$.
\end{lemma}
\begin{proof} Since, by the definition of $|\bb|$, 
$\bb(f)\leq |\rmD f||\bb|$ and since H\"older's inequality gives
$$
\int_X \bb(f)\di\meas=\int_X\Gamma(f,f)\di\meas\geq\int_X|\rmD f||\bb|\di\meas
$$
we obtain that
\begin{equation}\label{eq:first_pointwise}
\bb(f)=|\rmD f|^2=|\bb|^2 \qquad\text{$\meas$-a.e.~in $X$.}
\end{equation}
Now, given any $g\in H^{1,2}(X,\dist,\meas)$, the pointwise equalities 
\eqref{eq:first_pointwise} in combination with the
pointwise inequalities
$$
|\bb(f)|^2+2\eps \bb(f)\bb(g)+o(\eps)=|\bb(f+\eps g)|^2\leq
|\bb|^2|\rmD (f+\eps g)|^2=
|\bb|^2|\rmD f|^2+2\eps|\bb|^2\Gamma(f,g)+o(\eps)
$$
give the result.\qedhere
\end{proof}

\begin{definition}[Mosco convergence]
We say that the Cheeger energies $\Ch_n=\Ch_{\meas_n}$ Mosco converge to $\Ch$ if both the 
following conditions hold:
\begin{itemize}
\item[(a)] (\emph{Weak-$\liminf$}). For every $f_n \in L^2(X,\meas_n)$ $L^2$-weakly converging to $f\in L^2(X,\meas)$, one has
\[ \Ch(f)\le \liminf_{n \to \infty} \Ch_n(f_n).\]
\item[(b)] (\emph{Strong-$\limsup$}). For every $f \in L^2(X,\meas)$ there exist $f_n\in L^2(X,\meas_n)$, $L^2$-strongly converging to $f$ with
\begin{equation*}
\Ch(f)=\lim_{n \to \infty} \Ch_n(f_n).
\end{equation*}
\end{itemize}
\end{definition}

Such a convergence holds for example if $(X,\dist,\meas_n)$ are $\RCD(K,\infty)$ spaces with $\meas_n(B_r(\bar x))\leq c_1e^{c_2r^2}$, 
see \cite[Theorem~6.8]{GigliMondinoSavare13}. 

In the sequel we also say that $f_n\in H^{1,2}(X,\dist,\meas_n)$ are weakly
convergent in $H^{1,2}$ to $f\in H^{1,2}(X,\dist,\meas)$ if $f_n\to f$ weakly in $L^2$ and $\sup_n\Ch_n(f_n)$ is finite. Strong convergence
in $H^{1,2}$ is defined by requiring strong $L^2$ convergence of $f_n$ to $f$, and $\Ch(f)=\lim_n\Ch_n(f_n)$. 
We are going to use this well-known consequence of Mosco convergence:
\begin{equation}\label{eq:Mosco_preliminary_bis}
\lim_{n\to\infty}\int_X\Gamma_n(f_n,g_n)\di\meas_n=\int_X\Gamma(f,g)\di\meas
\end{equation}
for any $f_n$ strongly convergent in $H^{1,2}$ to $f$ and all $g_n$ weakly convergent in $H^{1,2}$ to $g$.

Indeed, since $f_n + t g_n$ weakly converge in $H^{1,2}$ to $f+t g$ for all $t>0$, by Mosco convergence we have
\begin{equation*}
\begin{split}
\Ch(f) + 2t \int_X \Gamma(f, g) \di \meas &+ t^2 \Ch(g) = \Ch(f+ t g) \\
& \le \liminf_{n \to \infty}  \Ch_n(f_n+ t g_n) \\
& = \liminf_{n \to \infty}  \Ch_n(f_n) + 2t \int_X \Gamma_n(f_n, g_n) \di \meas_n + t^2 \Ch_n(g_n)\\
& \le  \Ch(f) +  2 t \liminf_{n\to\infty}\int_X \Gamma_n(f_n,g_n) \di \meas + t^2 \limsup_n \Ch_n(g_n).
\end{split}
\end{equation*}
Since $\sup_n \Ch_n(g_n)$ is finite, we may let $t \downarrow 0$ to deduce the $\liminf$ inequality in \eqref{eq:Mosco_preliminary_bis};
replacing $g$ by $-g$ gives \eqref{eq:Mosco_preliminary_bis}.

The following simple example shows that even when the functions $f_n$, $g_n$ are fixed, \eqref{eq:Mosco_preliminary_bis} might
not hold.

\begin{example}
Take $X=\setR^2$ endowed with the Euclidean distance, $f(x_1,x_2)=x_2$ and let 
$$
\meas_n=n\leb^2\res \bigl([0,1]\times [0,\tfrac 1n]\bigr),\qquad
\meas=\haus^1\res [0,1]\times\{0\}.
$$
Then, it is easily seen that $\Gamma_n(f,f)=1$, while $\Gamma(f,f)=0$.
\end{example}

The next theorem shows that any sequence $(f_n)$ strongly convergent in $H^{1,2}$ to $f$  
induces gradient derivations which are strongly converging to the gradient derivation of the limit function.

\begin{theorem}[Strong convergence of gradient derivations]\label{thm:strong-convergence-gradients}
Assume (a), (b)  of \autoref{sec:two_weak} on the limit structure $(X,\dist,\meas)$, 
that $\Ch_n$ are quadratic and that Mosco converge to $\Ch$.
Let $f_n \in H^{1,2}(X,\dist,\meas_n)$ be strongly convergent in $H^{1,2}$ to $f\in H^{1,2}(X,\dist,\meas)$. \\
Then the derivations $\bb_{f_n}$ strongly converge in $L^2$ to $\bb_f$. In addition, $\bb_{f_n}(g_n)$ weakly
converge in $L^2$ to $\bb_f(g)$ whenever $g_n\to g$ in $\Lipb(X)$ with $\cup_n\supp(g_n)$ bounded in $X$.
\end{theorem}

\begin{proof} Set $\bb_n=\bb_{f_n}$, $\bb=\bb_f$, and let $\Gamma_n$ be the bilinear form associated to $\Ch_n$.
From \eqref{eq:Mosco_preliminary_bis} we get
\begin{equation}\label{eq:Mosco_preliminary}
\lim_{n\to\infty}\int_X\bb_n(g)\di\meas_n=\int_X\bb(g)\di\meas
\qquad\forall g\in\Lip_\bs(X).
\end{equation}
By \autoref{thm:compa2} with $\Theta(z)= z^2$ there exist a subsequence $b_{n(k)}$ and $\cc\in \Der^2(X,\dist,\meas)$ such that
$\bb_{n(k)}\Weakto{P\Algebra_\bs}\cc$ and 
$$\int_X |\cc|^2 \di \meas \le \liminf_{k \to \infty} \int_X |\bb_{n(k)}|^2 \di \meas_{n(k)}\leq\int_X|\bb|^2\di\meas.$$ 

We fix a sequence of cut-off functions $\chi_R\in\Lip_\bs(X)$ with $0\leq \chi_R\leq 1$, $\Lip(\chi_R)\leq 2$ and
$\chi_R\equiv 1$ on $B_R(\bar x)$.
To show that $\bb = \cc$, let $h\in P_{\setQ_+}\Algebra_\bs$ and
let us notice that the Young inequality and the $L^2(X,\meas)$ integrability of $h$ easily give
\begin{equation}\label{eq:Mosco_preliminary_ter}
\limsup_{R\to\infty}\limsup_{n\to\infty}\int_Xh|\Gamma_n(f_n,\chi_R)|\di\meas_n=
\limsup_{R\to\infty}\limsup_{n\to\infty}\int_{X\setminus\bar{B}_R(\bar x)}h|\Gamma_n(f_n,\chi_R)|\di\meas_n=0.
\end{equation}
Now, taking \eqref{eq:Mosco_preliminary} with $g=h\chi_R$ and \eqref{eq:Mosco_preliminary_ter} into account, we can pass to the limit as $n\to\infty$
in the identity
$$
\int_X \bb_n(h)\chi_R\di\meas_n=\int_X\Gamma_n(f_n,h\chi_R)\di\meas_n-\int_Xh\Gamma_n(f_n,\chi_R)\di\meas_n
$$
to get $\int_X \cc(h)\chi_R\di \meas = \int_X \bb( h\chi_R) \di \meas+\omega_R$, with
$\omega_R\to 0$ as $R\to\infty$. Taking the limit w.r.t.~$R$, we obtain
$\int_X \bb(h)\di \meas = \int_X \cc(h) \di \meas$ for every $h\in P_{\setQ_+}\Algebra_\bs$.
This identity extends, by the density of $P_{\setQ_+}\Algebra_\bs$, to any $h\in H^{1,2}(X,\dist, \meas)$; then, 
\autoref{lem:Hodge} applies and gives $\bb= \cc$.

By sequential compactness, this proves that $\bb_n\Weakto{P\Algebra_\bs}\bb$.
Strong convergence finally follows from \autoref{thm:strong-stability-hilbert}, with $\Theta(z) = z^2$.

In order to prove the final part of the statement it is sufficient to apply the Leibniz formula, passing to the
limit as $n\to\infty$ in
$$
\int_X\bb_n(g_n)a\di\meas_n=\int_X\Gamma_n(f_n,g_na)\di\meas_n-\int_X\bb_n(a)g_n\di\meas_n
\qquad a\in P_{\setQ_+}\Algebra_\bs
$$
and using \eqref{eq:Mosco_preliminary_bis} and the $L^2$ convergence of $\bb_n(a)$ to $\bb(a)$. This
proves that $\bb_{f_n}(g_n)$ weakly converge to $\bb_f(g)$ in the duality with $P_{\setQ^+}\Algebra_\bs$, and since
this class of test functions uniquely determines the limit, the thesis follows.
\end{proof}

{\color{blue}
\begin{remark} In the setting of the analysis of Ricci limit spaces, S.~Honda deeply studied in \cite{Honda1}
notions of convergence for gradient derivations $\bb_{f_n}$ associated to $f_n$, by looking essentially at the weak convergence of 
of $\bb_{f_n}(g_n)$ to $\bb_f(g)$ when $g_n=d(\cdot,z_n)$, $g=d(\cdot,z)$, with $z_n\to z$ (see
in particular Definition 4.18 in \cite{Honda1}, and \cite{Honda2} for more general tensor fields). In this respect, the final part of the
statement of Theorem~\ref{thm:strong-convergence-gradients} shows the connection between ours and Honda's convergence, under
the assumption of strong convergence in $H^{1,2}$ of the $f_n$ to $f$. 
\end{remark}}

We conclude this section by providing two examples.

\begin{example}[Convergence of resolvents]\label{ex:resolvent} Under the assumptions of the previous theorem,
let $f_n \in L^2(X,\meas_n)$ be strongly convergent in $L^2$ to $f \in L^2(X,\meas)$, let $\lambda>0$ and set $u_n = (\lambda -\Delta_n)^{-1} f_n \in 
D(\Delta_n)$, so that $\Delta_n u_n = \lambda u_n -f_n$, and consider the gradient derivations $\bb_{u_n}$.

It is then known (see e.g., \cite[Corollary~6.10]{GigliMondinoSavare13}) that Mosco convergence entails $L^2$-strong convergence of $u_n$ to $u =  (\lambda -\Delta)^{-1} f \in  D(\Delta)$, as well as $\lim_n\Ch_n(u_n) = \Ch(u)$. We may choose $\Theta(z) =z^2$, to fulfil the assumptions of \autoref{thm:strong-convergence-gradients} and we deduce strong convergence of $\bb_{u_n}$ to 
$\bb_u$.
\end{example}

\begin{example}[Laplacian eigenvalues]\label{ex:eigenfunctions} Let us assume that $(X,\dist,\meas_n)$ are $\RCD(K,\infty)$ spaces, so that in
particular the assumptions of the previous theorem hold. In this case, let us consider normalized eigenfunctions of (minus) the Laplacian operators, 
i.e.~$u_n \in D(\Delta_n)$ which satisfy $\int_X u_n^2 \di \meas_n = 1$, and $-\Delta_n u_n =\lambda_n u_n$, for some $\lambda_n\in\setR^+$.
Assuming in addition that either $K>0$ or all $\meas_n$ are probability measures,
using min-max arguments and Mosco convergence of the Cheeger energies, it has been proved in \cite[Theorem~7.8]{GigliMondinoSavare13} that, representing 
the discrete spectra of $-\Delta_n$ as $(\lambda^k_n)_{k \ge 0}$ (in non-decreasing order), for each $k \ge 0$ the eigenvalues $\lambda^k_n$ converge 
as $n\to\infty$ to the $k$-th eigenvalue of $\lambda^k$ of $-\Delta$, and that the associated eigenfunctions $u_n^k$ $L^2$-strongly converge to 
a corresponding eigenfunction $u^k$, possibly extracting a subsequence (there is no need to extract subsequences if the limit eigenvalue is simple).

Since $\Ch_n(u_n^k) = \lambda_n^k \to  \lambda^k$, 
by \autoref{thm:strong-convergence-gradients} we deduce the strong convergence of the gradient derivations $\bb_{u_n^k}$.
\end{example}

\part{Flows associated to derivations and their convergence}\label{pt:flows}

\section{Continuity equations and flows associated to derivations}\label{sec:contiflows}

Given $(\bb_t)_{t\in (0,T)}\subset\Der^1_\loc(X,\dist,\meas)$, we now consider the continuity equation
\begin{equation}\label{eq:ce}\tag{CE}
\left\{\begin{aligned}
&\didi{t}\mu_t + \div(\bb_t\mu_t) = 0, \\
&\mu_0 = \bar\mu
\end{aligned}\right.
\end{equation}
and its weak formulation in the space of probability measures $\mu_t=u_t\meas$ absolutely continuous w.r.t.~$\meas$.

\begin{definition}[Weak solutions to the continuity equation]\label{def:ce}
Let $\mu_t=u_t\meas\in\Prob(X)$, $t\in (0,T)$, and $\bar\mu\in\Prob(X)$. We say that 
$\mu_t$ is a solution to \eqref{eq:ce} if
\begin{equation}\label{eq:ce-integrability}
\int_0^T \int_X |\bb_t(f)| \di\mu_t \di t < \infty\qquad\forall f\in\Lip_\bs(X)
\end{equation}
and for every $f\in\Lip_\bs(X)$ and $\chi\in C^1_c([0,T))$ one has
\begin{equation}\label{eq:ce1}
-\int_0^T\chi'(t)\int_X f \di\mu_t\di t = \int_0^T\chi(t)\int_X \bb_t(f) \di\mu_t \di t+\chi(0)\int_Xf\di\bar\mu.
\end{equation}
\end{definition}

\begin{remark}[Different classes of integrability and test functions] {\rm Under the stronger assumption $\mu_t\leq C\meas$, the weak formulation
of \eqref{eq:ce} makes sense for all test functions $f\in\Lipb(X)$ assuming the condition 
\begin{equation}\label{eq:ce-integrability-bis}
\int_0^T\|\bb_t(f)\|_{L^1+L^\infty(X,\meas)}\di t<\infty\qquad\forall f\in\Lipb(X),
\end{equation}
somehow weaker than \eqref{eq:ce-integrability} (a weaker norm, but a larger class of functions). If
one has 
\begin{equation}\label{eq:ce-integrability-ter}
\int_0^T\int_X\Theta(|\bb_t|)\di\meas\di t<\infty
\end{equation}
for some $\Theta:[0,\infty)\to [0,\infty]$ with more than linear growth at infinity, then the inequality
$z\leq c(1+\Theta(z))$ immediately yields that \eqref{eq:ce-integrability-bis} holds. 
Furthermore, if we assume \eqref{eq:ce-integrability-ter} and $\mu_t\leq C\meas$, it is easy
by a truncation argument to pass in the weak formulation of \eqref{eq:ce} from $\Lip_\bs(X)$ to $\Lipb(X)$.
}\end{remark}

We say that a solution $\mu_t$ to \eqref{eq:ce} is weakly continuous if $t\mapsto \int_X f\di\mu_t$ is 
continuous in $(0,T)$ for all $f\in\Lip_\bs(X)$; under this assumption, because of \eqref{eq:ce-integrability}, 
the map is absolutely continuous and \eqref{eq:ce1} can be written in the equivalent form
$$
\didi{t}\int_X f\di\mu_t=\int_X\bb_t(f)\di\mu_t
\quad\text{for $\leb^1$-a.e.~$t\in (0,T)$},\qquad
\lim_{t\downarrow 0}\int_X f\di\mu_t=\int_X f\di\bar\mu.
$$
A natural class of examples of weakly continuous solutions to \eqref{eq:ce} is given by $\mu_t=(\ev_t)_\#\ppi$
with $\ppi\in \Prob(C([0,T];X))$, see \autoref{rem:pi_induce_conti} below for more details.

\begin{definition}[Regular flow relative to $\bb$]\label{dflow}
Let $\XX:[0,T]\times X\to X$ be a Borel map. We say that $\XX$ is a regular flow relative to $\bb_t$ if:
\begin{itemize}
\item[(a)] for some constant $C=C(\XX,\meas)$, one has $\XX(t,\plchldr)_\#\meas\leq C\meas$ for all $t\in [0,T]$;
\item[(b)] $\XX(\plchldr,x)\in AC([0,T];X)$ and $\XX(0,x)=x$ for $\meas$-a.e.~$x\in X$ and, for all $f\in\Lip_\bs(X)$, one has 
$$
\didi{t} f\circ\XX(t,x)=\bb_t(f)(\XX(t,x))\qquad\text{for $\leb^1\times\meas$-a.e.~$(t,x)\in (0,T)\times X$.}
$$
\end{itemize}
\end{definition}

The property of being a regular flow is independent of the choice of representatives $\bb_t(f)$; to see this,
notice that the condition $\XX(t,\plchldr)_\#\meas\ll\meas$ for $\leb^1$-a.e.~$t\in (0,T)$, weaker than (a), and Fubini's theorem give that the set
$$
\left\{(t,x)\in (0,T)\times X:\ (t,\XX(t,x))\in N\right\}
$$ 
is $\leb^1\times\meas$-negligible
for any $\leb^1\times\meas$-negligible set $N\subset (0,T)\times X$.

In connection with the stability analysis, the previous definition needs to be extended in order to cover
generalized flows, namely flows where  branching behaviour is allowed. 

\begin{definition}[Regular generalized flow relative to $\bb$]\label{dregflow}
Let $\ppi\in\Prob(C([0,T];X))$. We say that $\ppi$ is a regular flow relative to $\bb_t$ if:
\begin{itemize}
\item[(a)] for some constant $C=C(\ppi,\meas)$, one has $(\ev_t)_\#\ppi\leq C\meas$ for all $t\in [0,T]$;
\item[(b)] $\ppi$ is concentrated on $AC([0,T];X)$ and for all $f\in\Lip_\bs(X)$ one has
\begin{equation}\label{eq:cond_b}
\didi{t} f\circ\gamma(t)=\bb_t(f)(\gamma(t))\,\,\,\text{for $\leb^1\times\ppi$-a.e.~$(t,\gamma)\in (0,T)\times C([0,T];X)$.}
\end{equation}
\end{itemize}
\end{definition}

As for \autoref{dflow}, it is the regularity condition (a), even in the weakened 
form $(\ev_t)_\#\ppi\ll\meas$ for $\leb^1$-a.e.~ $t\in (0,T)$,
that ensures that property (b) above is independent of the
choice of representatives of $\bb_t(f)$. Obviously any regular flow $\XX$ induces regular generalized flows $\ppi$, given by $\Sigma_\# (f\meas)$, where
$f$ is the density of a probability measure and $\Sigma:X\to C([0,T];X)$ is given by $\Sigma(x)=\XX(\plchldr,x)$. The converse holds if $\meas\ll (\ev_0)_\#\ppi$ and
if the conditional probability measures $\ppi_x$ in $C([0,T];X)$ induced by $\ev_0$ are Dirac masses $\meas$-a.e.~in $X$; indeed, if this is
the case, setting $\ppi_x=\delta_{\{\XX(\plchldr,x)\}}$, we recover $\XX$. 

\begin{remark}\label{rem:pi_induce_conti}{\rm
Under the mild integrability assumption \eqref{eq:ce-integrability} on $\bb_t$,
property \eqref{eq:cond_b} in \autoref{dregflow} and the continuity equation are closely related.
More precisely, if $\ppi\in\Prob(C([0,T];X))$ satisfies $(\ev_t)_\#\ppi\ll\meas$ for $\leb^1$-a.e.~$t\in (0,T)$ (weaker than (a)) and
\eqref{eq:cond_b}, then its marginals $\mu_t=(\ev_t)_\#\ppi$, $t\in [0,T]$, are weakly continuous in time and satisfy \eqref{eq:ce}.
Indeed, fix $f\in\Lip_\bs(X)$; it is clear that $t\mapsto\int_X f\di\mu_t$ is continuous. Moreover, by the Fubini-Tonelli theorem, for $\leb^1$-a.e.~$t\in (0,T)$ one has
$$
\didi{t} f\circ\gamma(t)=\bb_t(f)(\gamma(t))\qquad\text{for $\ppi$-a.e.~$\gamma\in C([0,T];X)$}
$$
and, by integration, one obtains
$$
\didi{t} \int_X f\di\mu_t=\didi{t}\int f(\gamma(t))\di\ppi(\gamma)=\int\bb_t(f)(\gamma(t))\di\ppi(\gamma)=
\int_X \bb_t(f)\di\mu_t.
$$
}\end{remark}

The following elementary criterion will be useful and provides a converse to \autoref{rem:pi_induce_conti}; roughly speaking, it links the validity of
the continuity equations for arbitrary modifications $g\ppi$,  $g \in \Cb(C([0,T]; X))$, to the property of being concentrated on solutions to the ODE,
in the weak sense expressed by condition (b).

\begin{proposition}[Concentration criterion]\label{prop:conc}
Let $\ppi\in\Prob(C([0,T];X))$ be concentrated on $AC([0,T];X)$ with $(\ev_t)_\#\ppi\ll\meas$ for $\leb^1$-a.e.~$t\in (0,T)$ and assume
that $\bb_t$ satisfy \eqref{eq:ce-integrability}.
Then, the following properties are equivalent:
\begin{itemize}
\item[(a)] for all $f \in \Lip_\bs(X)$, \eqref{eq:cond_b} holds;
\item[(b)] for all $g \in \Cb(C([0,T]; X))$, with $g \ppi \in\Prob(C([0,T];X))$,  the curve
$$
\mu^g_t:=(\ev_t)_\# (g\ppi), \quad \text{$t \in [0,T]$,}
$$
solves the continuity equation \eqref{eq:ce} with $\bar\mu=\mu^g_0$.
\end{itemize}
\end{proposition}

\begin{proof} We already proved the implication from  \eqref{eq:cond_b}  to the continuity equation. To show the converse, 
fix $g$ as in (b). For all $g\in\Lip_\bs(X)$, by integration in an interval $[s,t]\subset [0,T]$ we get
$$
\int_X f\di\mu^g_t-\int_X f\di\mu^g_s=\int_s^t\int_X\bb_r(f)\di\mu^g_r\di r, 
$$
so that
$$
\int f(\gamma(t))-f(\gamma(s)) \di g\ppi (\gamma)=\int \int_s^t\bb_r(f)(\gamma(r))\di r \di g \ppi(\gamma),
$$
i.e., $\int \left[ f(\gamma(t))-f(\gamma(s)) -\int_s^t\bb_r(f)(\gamma(r))\di r \right] g(\gamma) \di \ppi (\gamma)=0$.
Since $g$ varies in a sufficiently large class,  we obtain that $f(\gamma(t))-f(\gamma(s))=\int_s^t\bb_r(f)(\gamma(r))\di r$
for $\ppi$-a.e.~$\gamma$, for $s$ and $t$ fixed. Since $f\circ\gamma\in AC([0,T])$ for $\ppi$-a.e.~$\gamma$, by 
letting $s$ and $t$ vary in $\setQ\cap [0,T]$ we obtain, via a density argument
$$
f(\gamma(t))-f(\gamma(s))=\int_s^t\bb_r(f)(\gamma(r))\di r\quad\text{for all $s,\,t\in [0,T]$ with $s\leq t$, for $\ppi$-a.e.~$\gamma$,}
$$
which yields \eqref{eq:cond_b}.
\end{proof}

\begin{remark}[Metric speed of a generalized flow]\label{rem:metric-speed}
Let us recall that $|\bb_t|$ provides an upper bound for the metric speed of curves selected by a generalized flow: indeed, 
\cite[Lemma 7.4]{AmbrosioTrevisan14} gives the inequality $\abs{\dot \gamma} (t) \le \abs{\bb_t}(\gamma(t))$ $\leb^1$-a.e.~in $(0,T)$, 
for $\ppi$-a.e.~$\gamma$. 
\end{remark}

\begin{proposition}[Tightness for generalized flows]\label{prop:tightness}
Assume that $\ppi_n\in\Prob(C([0,T];X))$ are regular generalized flows relative to $\bb_{n,t}$ such that
\begin{equation}\label{eq:no_dispersion}
\lim_{R\to\infty}\sup_n\,(\ev_0)_\#\ppi_n\bigl(X\setminus B_R(\bar x)\bigr)=0
\end{equation} 
for some $\bar x\in X$,
\begin{equation*}
 \sup_n C(\ppi_n,\meas_n)<\infty \quad \text{ and} \quad
\sup_n \int_0^T\int_X\Theta(|\bb_{n,t}|)\di\meas_n \di t < \infty,
\end{equation*}
with $\Theta:[0,\infty)\to [0,\infty]$ having more than linear growth at infinity. Then
the family $\{\ppi_n\}$ is tight in $\Prob(C([0,T];X))$, and any limit point $\ppi\in\Prob(C([0,T];X))$ is concentrated  on $AC([0,T];X)$.
\end{proposition}

\begin{proof} In order to prove tightness,
we build a coercive functional $\Psi:C([0,T];X)\to[0,\infty]$ (i.e., a map with relatively compact sublevel sets) such that
\[
\sup_n\int_{C([0,T];X)} \Psi(\gamma) \di\ppi_n(\gamma) < \infty.
\]
Let $R_i\uparrow\infty$ with $\meas(\partial B_{R_i}(\bar x))=0$. From the convergence
of $\meas_n\res \bar{B}_{R_i}(\bar x)$ to $\meas\res\bar{B}_{R_i}(\bar x)$ we obtain coercive functions
$\psi_i:\bar{B}_{R_i}(\bar x)\to [0,\infty]$ such that
$$
\sup_n\int_{\bar{B}_{R_i}(\bar x)}\psi_i\di\meas_n\leq 2^{-i}.
$$
Setting $\psi_i=0$ on $X\setminus\bar{B}_{R_i}(\bar x)$ and $\psi=\sum_{i\geq 1}\psi_i$, we obtain
$$
\sup_n\int_X\psi\di\meas_n\leq 1
$$
and, since $\psi\geq\psi_i$ and $R_i\to\infty$, all sets $\{\psi\leq t\}\cap\bar{B}_R(\bar x)$ are relatively compact in $X$. 
Using \eqref{eq:no_dispersion} we can also find $\phi:X\to [0,\infty]$ with $\phi(x)\to\infty$ as $\dist(x,\bar x)\to\infty$
and $\sup_n\int_X\phi(\gamma(0))\di\ppi_n(\gamma)<\infty$.

Now fix a countable dense set $\{t_j\}$ in $[0,T]$ and define
\[
\Psi(\gamma) = \phi(\gamma(0))+\sum_{j\in\setN} 2^{-j} \psi(\gamma(t_j)) +
\int_0^T \Theta(\abs{\dot\gamma}(t)) \di t
\]
if $\gamma \in AC([0,T];X)$, $\Psi(\gamma)=\infty$ otherwise. Thanks to the above-mentioned local compactness
property of the sublevel sets of $\psi$, it is easily seen  (as in Ascoli-Arzel\`a's theorem)
that $\Psi$ is coercive. It is now clear, using the condition $(\ev_{t_j})_\#\ppi_n\leq C(\ppi_n,\meas_n)\meas_n$,  that 
$$
\sup_n \int_{C([0,T];X)}\bigg[\phi(\gamma(0))+\sum_{j\in\setN} 2^{-j} \psi(\gamma(t_j))\bigg]\di\ppi_n(\gamma)<\infty.
$$
In addition
\begin{eqnarray*}
\int_{C([0,T];X)} \int_0^T \Theta(\abs{\dot\gamma}(t)) \di t\di\ppi_n(\gamma) &\leq& 
	\int_0^T \int_{C([0,T];X)} \Theta(\abs{\bb_{n,t}}(\gamma(t)) )  \di \ppi_n(\gamma)\di t \\
&\leq& C (\ppi_n,\meas_n)\int_0^T \int_X \Theta(\abs{\bb_{n,t}}) \di \meas_n \di t,
\end{eqnarray*}
where the first inequality above follows from \autoref{rem:metric-speed}.
\end{proof}

In the following proposition we prove that limits of flows relative to $\bb_{n,t}$ are
flows relative to $\bb_t$, if $\bb_{n,t}$ strongly converge according to \eqref{eq:strong-stability-hilbert-bis}. This condition,
stated in the minimal form needed for the validity of the proof, is in many cases implied by the notion of
strong convergence of the previous sections, for instance in $\RCD(K,\infty)$ spaces the class
$\mathcal D=P_{\setQ^+}\Algebra_\bs$ is dense w.r.t.~$\meas$-flat convergence and one can use
the convergence in measure of $\bb_{n,t}(f)$ to $\bb_t(f)$ for $f\in\mathcal D$, together with uniform
bounds in $L^1+L^\infty$, to prove \eqref{eq:strong-stability-hilbert-bis}. We do not discuss in generality this point,
referring to the specific examples discussed in \autoref{sec:stabile_flow}.

\begin{proposition}[Closure theorem]\label{prop:closure}
Assume that $\ppi_n\in\Prob(C([0,T];X))$ are regular generalized flows relative to $\bb_{n,t}$ and that:
\begin{itemize}
\item[(i)] $\sup_nC(\ppi_n,\meas_n)<\infty$ and $\ppi_n$ weakly converge to $\ppi\in\Prob(C([0,T];X))$, with
$\ppi$ concentrated on $AC([0,T];X)$;
\item[(ii)] for all $f$ in a class $\mathcal D\subset\Lipb(X)$, dense w.r.t.~$\meas$-flat convergence,  one has
\begin{equation}\label{eq:strong-stability-hilbert-bis}
\lim_{n\to\infty}\int_0^T\int_X \bb_{n,t}(f)v_n\di\meas_n\di t=\int_0^T\int_X\bb_t(f)v\di\meas\di t
\end{equation}
whenever $0\leq v_n\leq C<\infty$ and $v_n(t,\plchldr)\meas_n\in\Prob(X)$ weakly converge to
$v(t,\plchldr)\meas\in\Prob(X)$ for all $t\in (0,T)$;
\item[(iii)] $\int_0^T\|\bb_t(f)\|_{L^1+L^\infty(X,\meas)}\di t<\infty$ and $\sup_n\int_0^T\|\bb_{n,t}(f)\|_{L^1+L^\infty(X,\meas_n)}\di t<\infty$ for all $f\in\Lipb(X)$;
\item[(iv)] either $H^{1,2}(X,\dist,\meas)$ is reflexive, or $\int_0^T\int_{B_R(\bar x)}|\div\bb_t|\di\meas\di t<\infty$ for all $R>0$.
\end{itemize}
Then $\ppi$ is a regular generalized flow relative to $\bb_t$.
\end{proposition}
\begin{proof}
Fix $g \in \Cb( C([0,T];X) )$ with $g\ppi \in \Prob( C([0,T]; X) )$,  set $g_n = g/ \int g \di \ppi_n$, so that, for $n$ large enough, the measures $\ppi_n^{g_n}:=g_n\ppi_n   \in \Prob( C([0,T]; X) )$ are well-defined and weakly converge to $\ppi^g:= g\ppi$.

By \autoref{prop:conc}, applied with $\pi_n$ and $g_n$, the marginal measures $\mu_{n,t}^{g_n}:=(\ev_t)_\#\ppi_n^{g_n}$ solve \eqref{eq:ce} with $\bar\mu=\mu_{n,0}^{g_n}$, have uniformly bounded (w.r.t.~$t\in [0,T]$ and $n$ sufficiently large) densities w.r.t.\ $\meas_n$,
\[ u_{n,t}^{g_n} \le \sup_n  C(\ppi_n,\meas_n) \norm{g_n}_\infty \] and weakly converge to $\mu^g_t:=(\ev_t)_\#\ppi^g$, since $\ev_t$ is a continuous map. 
Passing to the limit as $n\to\infty$ in the weak formulation
$$
-\int_0^T\chi'(t)\int_X f \di\mu^{g_n}_{n,t}\di t = \int_0^T\chi(t)\int_X \bb_{n,t}(f)u^{g_n}_{n,t} \di\meas_n \di t+\chi(0)\int_Xf\di \mu_{n,0}^{g_n}.
$$ 
with $f\in\mathcal D$ and using the convergence assumption with $v_n(t,x)=T^{-1} u^{g_n}_{n,t}(x)$ we get
\begin{equation*}
-\int_0^T\chi'(t)\int_X f \di\mu^g_t\di t = \int_0^T\chi(t)\int_X \bb_t(f) \di\mu^g_t \di t+\chi(0)\int_Xf\di\mu^g_0.
\end{equation*}
for all $f\in\mathcal D$. By an easy approximation based on assumption (iv) and
\autoref{lem:continuity_div}, the density of $\mathcal D$ gives that $\mu^g_t$ solve 
\eqref{eq:ce} with $\bar\mu=\mu^g_0$. Again by \autoref{prop:conc}, we obtain that $\ppi$ is a regular generalized flow relative to $\bb$.
\end{proof}

\begin{remark} In \autoref{prop:closure} the strong convergence property of $\bb_{n,t}$ expressed by
\eqref{eq:strong-stability-hilbert-bis} can not be replaced, in general, by 
weak convergence together with convergence of the norms. In fact, we have the following simple counterexample. 
Consider in $(\setR^3,\norm{\plchldr}_\infty, \leb^3)$ the vector fields
\[
\bb_n(x) = w_n(x_1)e_2+e_3
\]
with (weak convergences are meant in the weak-$*$ topology of $L^\infty$)
\[
w_n:\setR\to\setR, \qquad w_n \weakto 0, \qquad \abs{w_n}=1, \qquad w_n^+ \weakto \frac12.
\]
For instance, we can take $w_n(t)=\sign\sin(2^n t)$. Then we have that $\bb_n\weakto e_3$ and $\lim_n|\bb_n|=|e_3|$. 
Let us call $\XX_n$ the flow relative to $\bb_n$. Then, given any probability density $u$,
we have the following convergence of the associated generalized 
flows $\eta_n=X_n(\plchldr,  x)_\# (u(x)\leb^3)$:
\[
\int_{\setR^3} \delta_{\XX_n(\plchldr, x)} u(x) \di x
\Weakto{\Prob(C([0,T];\, \setR^3))}
\int_{\setR^3} \frac12\left(\delta_{\gamma^+(\plchldr, x)}+\delta_{\gamma^-(\plchldr, x)}\right)u(x) \di x
\]
where
\[
\gamma^\pm(t,x) = \bigl( x_1, x_2 \pm\frac 12t, x_3+t \bigr).
\]
Therefore, we have that the generalized flows converge in $\Prob(C([0,T];\setR^3))$ to a probability measure
that is not a generalized flow relative to the weak limit $e_3$.

This example raises a problem about possible extensions of \autoref{thm:strong-stability-hilbert}  to the case
when the norm on derivations is not strictly convex, since oscillations in the vector fields might not be detected by the convergence
of the norms.
\end{remark}

We end this section with the following stability result for regular flows, under the assumption that for the limit vector field
any regular generalized flow is induced by a regular flow, see \eqref{eq:generalized-flow-induced}. Notice that with a similar
proof a similar convergence result holds if we replace $\XX_n$ by regular generalized flows $\ppi_n$.

\begin{proposition}[Stability for regular flows, general case]\label{prop:stability-general}
Assume that $\XX_n$ are regular flows relative to $\bb_{n,t}$ with $\bb_{n,t}$ strongly convergent to $\bb_t$
according to \eqref{eq:strong-stability-hilbert-bis} of \autoref{prop:closure},
\[ \sup_n C(\XX_n,\meas_n)<\infty \quad \text{and} \quad
\sup_n \int_0^T\int_X\Theta(|\bb_{n,t}|)\di\meas_n\di t < \infty,
\]
and $\Theta:[0,\infty)\to [0,\infty]$ having more than linear growth at infinity. Assume also that
any regular generalized flow $\ppi$ relative to $\bb$ is induced by a regular flow $\XX$ relative to $\bb$, i.e.~
\begin{equation} \label{eq:generalized-flow-induced} \ppi = \int_X \XX(\plchldr, x) \di (\ev_0)_\# \ppi.\end{equation}
Then, $\XX_n: X \to C([0,T]; X)$ converge in measure towards $\XX: X \to C([0,T]; X)$.
\end{proposition}


\begin{proof}
To deduce convergence in measure, we rely on \autoref{prop:convergence-in-measure}, with $Y=C([0,T]; X)$. Let $v \in \Cbs(X)$ nonnegative with 
$\int_X v \di \meas =1$, let $\bar{x} \in X$  and $\bar{R}>0$ large enough so that $\supp(v) \subset B_{\bar R}(\bar{x})$. We define $c_n := \int_X v \di \meas_n$, which converge to $1$ as $n \to \infty$,  and, for $n$ large enough, $\ppi_n  := c_n^{-1}\int_X \XX(\plchldr, x) v(x) \di \meas_n  \in\Prob(C([0,T]; X))$. Then, $\ppi_n$ is  a regular generalized flow, relative to $\bb_n$, with $C(\ppi_n, \meas_n) \le c_n^{-1} C(\XX_n, \meas_n) \sup |v|$, hence uniformly bounded in $n$.

By \autoref{prop:tightness}, the family $\{\ppi_n\}$ is tight, in $\Prob(C([0,T]; X))$: indeed, the only non trivial condition to check is \eqref{eq:no_dispersion}, but that expression  is $0$ for $R > \bar{R}$. Next, by \autoref{prop:closure}, any limit point $\ppi$ is a regular generalized flow relative to $\bb$: in this case, we have to check  condition \emph{(iii)} only, which follows from the inequality 
\[ |b_{n,t}(f)| \le  |b_{n,t}| \Lip(f) \le c\Lip(f)(1+\Theta(|\bb_{n,t}|), 
\]
where $c>0$ is some constant such that $|z| \le c(1+ \Theta(z))$ for all $z\in [0,\infty)$. Hence, our assumptions entail then that $\ppi$ can be written as in \eqref{eq:generalized-flow-induced}.

To apply \autoref{prop:convergence-in-measure} with $Y=C([0,T]; X)$, let $\Phi: Y \to\setR$ and $g: X \to\setR$ 
be bounded and continuous functions. Then,
\begin{equation*}\begin{split}
\int_{X \times Y} g(x)\Phi(\gamma) \di(\Id \times \XX_n)_\#(v\meas_n)(x,\gamma) & =
\int_X g(x),\Phi(\XX_n(x,\plchldr))  v(x) \di \meas_n(x) \\
\text{(since $X_n(x,0)=x$ $\meas_n$-a.e.~on $X$)}\quad \quad& = \int_X g(\XX_n(x,0))\Phi(\XX_n(x,\plchldr))  v(x) \di \meas_n(x)  \\
& = \int_Y g(\ev_0(\gamma))\Phi(\gamma)  \di \ppi_n(\gamma)  \\ \text{(as $n \to \infty$)}\quad \quad & 
\to \int_Y g(\ev_0(\gamma))\Phi(\gamma)  \di \ppi(\gamma)\\
\text{(by the representation \eqref{eq:generalized-flow-induced})}\quad\quad & = \int_X g(\XX(x,0))\Phi(\XX(x,\plchldr))  v(x) \di \meas(x)\\
& = \int_{X \times Y } g(x) \Phi( \gamma )  \di (\Id \times \XX)_\# (v\meas)(x,\gamma),
\end{split}\end{equation*}
where the limit as $n\to \infty$ follows from the fact that $g(\ev_0)\Phi$ is a bounded continuous function on $Y$.
\end{proof}

\section{Convergence of flows under Ricci curvature bounds}\label{sec:stabile_flow}

Let us recall the following well-posedness result, from \cite{AmbrosioTrevisan14}, that we report here in a form suitable for our purposes.  
A crucial assumption, leading to well-posedness of the continuity equation, and then to the existence and uniqueness of the
regular flow, is on the ``deformation'' $D^\sym \bb$ of $\bb$ (corresponding to the symmetric part of derivative in the Euclidean
case), which can be defined in a suitable weak sense \cite[Definition~5.2]{AmbrosioTrevisan14}. For example, if $\bb(\plchldr) = (\omega +c) \Gamma(f,\plchldr)$, for $f \in D(\Delta)$, $\omega\in H^{1,2}(X,\dist,\meas) \cap L^\infty(X,\meas)$, $c\in\setR$, then \cite[Theorem~6.7]{AmbrosioTrevisan14} gives $\| D^\sym \bb\|_{4,4}<\infty$.

\begin{theorem}[Well-posedness of flows]\label{thm:well-posedness-rcd}
Let $(X,\dist,\meas)$ be a $\RCD(K,\infty)$ space, for some $K\in\setR$ and let $\bb=(\bb_t)_{t\in (0,T)}$ be a Borel time-dependent derivation with
\[ |\bb|,  \div \bb \in L^1_t(L^2(X,\meas)+L^\infty(X,\meas)), \quad (\div \bb)^- \in L^1_t(L^\infty(X,\meas)),\]
\[ \text{and} \quad \|D^\sym\bb_t\|_{4,4} \in L^1(0,T).\]
Then, there exists a unique regular flow $\XX$ relative to $\bb$ and every regular generalized flow $\ppi$ relative to $\bb$ is induced by $\XX$
as in \eqref{eq:generalized-flow-induced}. 
\end{theorem}
\begin{proof}
By \cite[Theorem~4.3, Theorem~5.4]{AmbrosioTrevisan14}, the continuity equation \eqref{eq:ce} associated to $\bb$ has existence and 
uniqueness of solutions in the class of weakly continuous solutions $\mu_t=u_t\meas\in\Prob(X)$, $t\in [0,T]$, with $\|u_t\|_{L^\infty(X,\meas)}$
bounded in $[0,T]$, so that  \cite[Theorem~8.3]{AmbrosioTrevisan14} applies, providing existence and uniqueness 
of a regular flow $\XX$ relative to $\bb$. The last statement follows from \cite[Theorem~8.4]{AmbrosioTrevisan14}.
\end{proof}

We state and prove our main result concerning stability of flows on converging $\RCD(K,\infty)$ metric measure spaces. 
Let us stress the fact that bounds on divergence and deformation are assumed only for the limit derivation.

\begin{theorem}[Stability of regular flows under curvature assumptions]\label{thm:stabflow_curv}
Let $\bb=(\bb_t)_{t\in (0,T)}\in L^1_t(\Der^1_\loc(X,\dist, \meas))$ as in \autoref{thm:well-posedness-rcd}.
For $n \ge 1$, let $\bb_n=(\bb_{n,t})_{t\in (0,T)}\in L^1_t(\Der^1_\loc(X,\dist, \meas_n))$ and assume that
$$
\lim_{n\to\infty}\int_0^T\chi(t)\int_X\bb_{n,t}(a)v\di\meas_n\di t=
\int_0^T\chi(t)\int_X\bb_t(a)v\di\meas\di t
$$
for all $\chi\in C_c(0,T),\,\,v\in\Cbs(X),\,\,a\in P_{\setQ_+}\Algebra_\bs$, and
\begin{equation*}
\limsup_{n\to\infty}\int_0^T\int_X\Theta(|\bb_{n,t}|)\di\meas_n\,\di t\leq
\int_0^T\int_X\Theta(|\bb|_t)\di\meas\,\di t<\infty
\end{equation*}
with $\Theta:[0,\infty)\to [0,\infty)$ strictly convex and having more than linear growth at infinity.

Let $\XX_n$ be regular flows relative to $\bb_n$, and $\XX$ be the ($\meas$-a.e.~unique) regular flow relative to $\bb$. If $\sup_{n} C(\XX_n, \meas_n) < \infty$, then $\XX_n:X \to C([0,T]; X)$ converge in measure towards $\XX:X \to C([0,T]; X)$. 
\end{theorem}

\begin{proof}
The thesis follows from \autoref{prop:stability-general}, with
$\mathcal D=P_{\setQ^+}\Algebra_\bs$. Indeed, it is sufficient to notice that, by \autoref{thm:compa2} and 
\autoref{thm:strong-stability-hilbert} we obtain strong convergence of $\bb_n$ to $\bb$, and by \autoref{thm:well-posedness-rcd}, we deduce that the representation \eqref{eq:generalized-flow-induced} holds true for any generalized flow $\ppi$ relative to $\bb$.
\end{proof}

\begin{example} [Convergence of flows associated to resolvents] In the setting of \autoref{ex:resolvent},
since the operator $(\lambda -\Delta_n)^{-1}$ are Markovian, if for some constant $c$ one has $|f_n| \le c$ $\meas_n$-a.e.\, then it follows $|u_n| \le c$ $\meas_n$-a.e.
and therefore $\div \bb_{u_n} =\lambda(u_n-f_n)\in L^\infty(X,\meas_n)$ uniformly w.r.t.~$n$. Now, by  \cite[Theorem~6.7]{AmbrosioTrevisan14}, one has $\| D^\sym\bb_{u_n}\|_{4,4} < \infty$. We are in a position to deduce that there exist the $\meas_n$-a.e.~unique regular flows $\XX_n$ relative to $\bb_{u_n}$ and that, as $n\to \infty$, 
$\XX_n$ converge in measure in $C([0,T];X)$ to the $\meas$-a.e.~unique flow $\XX$ relative to $\bb_u$.
\end{example}

\begin{example}[Convergence of flows associated to Laplacian eigenfunctions] In the setting of \autoref{ex:eigenfunctions},
for $k$ fixed set $\bb_n=\bb_{u_n^k}$ and $\bb=\bb_{u^k}$.
From \cite[Theorem~6.7]{AmbrosioTrevisan14} we obtain $\| D^\sym\bb_n\|_{4,4} < \infty$, but in order to obtain the validity of the
assumption $(\div \bb_n)^- \in L^\infty(X,\meas)$ in \autoref{thm:well-posedness-rcd}, the $\RCD(K,\infty)$ assumption is not sufficient, e.g., 
in case of Gaussian space where the eigenfunctions are Hermite polynomials. To obtain non trivial examples, we may restrict 
ourselves to metric measure spaces $(X,\dist,\meas_n)$ where the heat semigroup $P^n_t$ is ultracontractive, i.e., mapping for $t>0$ the space
$L^2(X,\meas_n)$ into $L^\infty(X,\meas_n)$.
Under this assumption, for $t>0$ one has
\[ \Delta_n u_n^k = -\lambda_n^k e^{\lambda_n^kt} P^n_t u_n^k \in L^\infty(X,\meas_n).\]
If we assume that quantitative ultracontractive bounds hold uniformly w.r.t.~$n$, then a unique regular flow $\XX_n$ relative 
to $\bb_n$ is defined and we deduce convergence in measure to the natural limit flow $\XX$, as $n \to \infty$. 
\end{example}

\part*{Appendix}
\addcontentsline{toc}{part}{Appendix}

\appendix

\section{Minimal relaxed slopes and Cheeger energy}\label{appendix:slopes}

In this section we recall basic facts about minimal relaxed slopes, Sobolev spaces and heat flow in metric measure spaces, see
\cite{AmbrosioGigliSavare13} and \cite{Gigli1} for a more systematic treatment of this topic. The Cheeger energy
$\Ch:L^2(X,\meas)\to [0,\infty]$ is the convex and lower semicontinuous functional defined as follows:
$$
\Ch(f):=\inf\left\{\liminf_{n\to\infty}\frac 12\int_X\Lipa ^2(f_n)\di\meas:\ \text{$f_n\in\Lipb(X)\cap L^2(X,\meas)$, $\|f_n-f\|_2\to 0$}\right\}. 
$$
The Sobolev space $H^{1,2}(X,\dist,\meas)$ is simply defined as the finiteness domain of $\Ch$. It can be proved
that $H^{1,2}(X,\dist,\meas)$ is Hilbert if $\Ch$ is quadratic, and reflexive if $(X,\dist)$ is doubling.
(see \cite{AmbrosioColomboDiMarino}).
 
In connection with the definition of $\Ch$, for all $f\in H^{1,2}(X,\dist,\meas)$ one can consider the collection $RS(f)$ of all functions in $L^2(X,\meas)$
larger than a weak $L^2(X,\meas)$ limit of $\Lipa (f_n)$, with $f_n\to f$ in $L^2(X,\meas)$. This collection describes a convex, closed and nonempty set, 
whose element with smallest $L^2(X,\meas)$ norm is called minimal relaxed slope and denoted by $|\rmD f|$. Because of the minimality
property, $|\rmD f|$ provides an integral representation to $\Ch$ and it is not hard to improve weak to strong convergence.  

\begin{theorem}\label{tapprox}
For all $f\in D(\Ch)$ one has
$$
\Ch(f)=\frac 12\int_X|\rmD f|^2\di\meas
$$
and there exist $f_n\in\Lipb(X)\cap L^2(X,\meas)$ with $f_n\to f$ in $L^2(X,\meas)$ and 
$\Lipa (f_n)\to |\rmD f|$ in $L^2(X,\meas)$.  In particular, if $H^{1,2}(X,\dist,\meas)$ is reflexive,
there exist $f_n\in\Lipb(X)\cap L^2(X,\meas)$ satisfying $f_n\to f$ in $L^2(X,\meas)$ and 
$|\rmD (f_n-f)|\to 0$ in $L^2(X,\meas)$. 
\end{theorem}

Most standard calculus rules can be proved, when dealing with minimal relaxed slopes. For the purposes of this paper the most relevant ones
are:
\begin{description}
\item[{\bf Locality on Borel sets.}] $|\rmD f|=|\rmD g|$ $\meas$-a.e.~on $\{f=g\}$ for all $f,\,g\in H^{1,2}(X,\dist,\meas)$;

\item[Pointwise minimality.] $|\rmD f|\leq g$ $\meas$-a.e.~for all $g\in RS(f)$;

\item[Degeneracy.] $|\rmD f|=0$ $\meas$-a.e.~on $f^{-1}(N)$ for all $f\in H^{1,2}(X,\dist,\meas)$ and all $\leb^1$-negligible $N \subset \setR$ Borel;

\item[Chain rule.] $|\rmD(\varphi\circ f)|=|\varphi'(f)||\rmD f|$ for all $f\in H^{1,2}(X,\dist,\meas)$ and all
$\varphi:\setR\to\setR$ Lipschitz with $\varphi(0)=0$.
\end{description}

Another object canonically associated to $\Ch$ and then to the metric measure structure is the heat flow $P_t$, defined as the
$L^2(X,\meas)$ gradient flow of $\Ch$, according to the Brezis-Komura theory of gradient flows of lower semicontinuous
functionals in Hilbert spaces, see for instance \cite{Brezis}. This theory provides a continuous contraction
semigroup. We shall use $P_t$ only in the case when $\Ch$ is quadratic, as a regularizing operator.
In this special case $P_t$ is also linear (and this property is equivalent to $\Ch$ being quadratic) and it is easily
seen that
$$
\lim_{t\downarrow 0} P_tf=f\qquad\text{for all $f\in H^{1,2}(X,\dist,\meas)$.}
$$

Finally, we describe the class of $\RCD(K,\infty)$ metric measure spaces of \cite{AmbrosioGigliSavare14}, where thanks to the lower
bounds on Ricci curvature even stronger properties of $P_t$ can be proved. We say that a metric measure space $(X,\dist,\meas)$
satisfying the growth bound (for some constants $c_1,\,c_2$ and some $\bar x\in X$)
\begin{equation}\label{eq:Grygorian}
\meas\bigl(B_r(\bar x)\bigr)\leq c_1 e^{c_2r^2}\qquad\forall r>0
\end{equation}
is a $\RCD(K,\infty)$ metric measure space, with $K\in\setR$, if $\Ch$ is quadratic and if, setting
$$
\Prob_2(X):=\left\{\mu\in\Prob(X):\ \int_X\dist^2(\bar x,x)\di\meas(x)<\infty\right\},
$$
the Relative Entropy Functional ${\rm Ent}(\mu):\Prob_2(X)\to\setR\cup\{\infty\}$ given by
$$
{\rm Ent}(\mu):=
\begin{cases}
\int_X\rho\log\rho\di\meas&\text{if $\mu=\rho\meas\ll\meas$,}
\\ 
\infty &\text{otherwise}
\end{cases}
$$
is $K$-convex along Wasserstein geodesics in $\Prob_2(X)$. See \cite{AmbrosioGigliSavare14} (dealing with finite reference measures),
\cite{AmbrosioGigliMondinoRajala} (for the $\sigma$-finite case) and \cite{AmbrosioGigliSavare15} for various characterizations
of this class of spaces. We quote here the following result, which essentially derives from the identification of $P_t$
as the gradient flow of ${\rm Ent}$ w.r.t.~the Wasserstein distance and the contractivity properties with respect to that distance.

\begin{proposition} In $\RCD(K,\infty)$ spaces $(X,\dist,\meas)$, for all $t>0$ the semigroup 
$P_t$ maps $L^2\cap L^\infty(X,\meas)$ to $\Cb(X)$ and if $f\in\Lip(X)\cap H^{1,2}(X,\dist,\meas)$ one has
$$
\Lipa (P_t f)\leq e^{-Kt} \sqrt{P_t|\rmD f|^2}\qquad\text{pointwise in $X$}.
$$
\end{proposition}

\section{An approximation result}\label{appendix:approximation}

In this section we improve \autoref{tapprox} by showing that the approximating sequence can be chosen
in the ``canonical'' algebra $\Algebra$ generated by the distance functions, and even in the subalgebra $\Algebra_\bs$. 
The deep reason why this is possible is the fact that this is the class of functions appearing in the Hopf-Lax formula, see \eqref{eq:Hopflax}.

\begin{theorem}\label{tapprox_refined}
For all $f\in D(\Ch)$ there exist $f_n\in\Algebra_\bs$ with $f_n\to f$ in $L^2(X,\meas)$ and 
$\Lipa (f_n)\to |\rmD f|$ in $L^2(X,\meas)$. In particular, if $H^{1,2}(X,\dist,\meas)$ is reflexive, one has
the existence of $f_n\in\Algebra_\bs$ satisfying $f_n\to f$ in $L^2(X,\meas)$ and 
$|\rmD (f_n-f)|\to 0$ in $L^2(X,\meas)$. 
\end{theorem}

\begin{proof} The proof follows closely the strategy developed in \cite{AmbrosioGigliSavare13}. We first build a variant of Cheeger's energy by
restricting the approximation to functions in $\Algebra_\bs$; this construction provides
a new minimal relaxed slope, that we denote by $|\rmD f|_\Algebra$ and that we prove to coincide with $|\rmD f|$ passing through
the notion of minimal $2$-weak upper gradient $|\rmD f|_w$. Let us outline the main steps, quoting them also from \cite{AmbrosioColomboDiMarino}, 
whose context is closer to the one of the present paper
(indeed, in \cite{AmbrosioGigliSavare13} more general metric measure structures, with possibly infinite distances, are allowed).

\noindent
{\bf Step 1.} (Construction of $\Ch_\Algebra$, $|\rmD f|_\Algebra$ and calculus rules). For all $f\in L^2(X,\meas)$ we define
$$
\Ch_\Algebra(f) := \inf\left\{\liminf_{n\to\infty}\frac 12\int_X\Lipa ^2(f_n)\di\meas:\ \text{$f_n\in\Algebra_\bs$, $f_n\to f$ in $L^2(X,\meas)$}\right\}. 
$$
It is immediate to check that $\Ch_\Algebra$ is a convex and $L^2(X,\meas)$-lower semicontinuous functional in $L^2(X,\meas)$,
with a dense domain (here we use that $\Algebra_\bs$ is dense in $L^2(X,\meas)$). It is also obvious by the definition that
\begin{equation}\label{eq:zuri1}
\Ch\leq\Ch_\Algebra
\end{equation}
where $\Ch$ is Cheeger's functional considered in the previous section.

As in  \cite[Section~4]{AmbrosioColomboDiMarino}, for all $f\in D(\Ch_\Algebra)$ one can define $RS_\Algebra(f)$ as the collection of all functions
$g\in L^2(X,\meas)$ which are larger than a weak $L^2(X,\meas)$ limit of $\Lipa (f_n)$, with $f_n\in\Algebra_\bs$ and $f_n\to f$ in $L^2(X,\meas)$. By
weak $L^2$-compactness, this set is not empty and it is not hard to show that it is closed and convex. Its element with minimal
$L^2(X,\meas)$ norm is defined to be $|\rmD f|_\Algebra$. Since $RS_\Algebra(f)\subset RS(f)$ for all $f\in D(\Ch_\Algebra)$, we can 
use the pointwise minimality property of $|\rmD f|$ to refine
\eqref{eq:zuri1} to a pointwise inequality:
\begin{equation}\label{eq:zuri11}
|\rmD f|\leq |\rmD f|_\Algebra\qquad\text{$\meas$-a.e.~in $X$ for all $f\in D(\Ch_\Algebra)$.}
\end{equation}

As we already mentioned in the previous section, in connection with $|\rmD f|$,
several properties of $|\rmD f|_\Algebra$ stem from the variational definition, in particular, 
\begin{itemize}
\item[(a)] there exist $f_n\in\Algebra_\bs$ with $f_n\to f$ in $L^2(X,\meas)$ and $\Lipa (f_n)\to |\rmD f|_\Algebra$ strongly in
$L^2(X,\meas)$;
\item[(b)] locality on Borel sets, pointwise minimality, degeneracy and chain rule;
\item[(c)] for all $g\in D(\Ch_\Algebra)$ where $\Ch_\Algebra$ has nonempty subdifferential one has
$$
-\int_X f\Delta_\Algebra g\di\meas\leq \int_X |\rmD f|_\Algebra|\rm D g|_\Algebra\di\meas,
$$
where $-\Delta_\Algebra g$ is the element with smallest norm in the subdifferential of $\Ch_\Algebra(g)$, with equality if
$f=\phi(g)$ with $\phi$ Lipschitz, nondecreasing, $\phi(0)=0$.
\end{itemize}

Notice that the proof of properties (b), which rests on the decomposability property of $RS_\Algebra$
$$
\chi g_1+(1-\chi)g_2\in RS_\Algebra(f)\qquad\text{for all $g_1,\,g_2\in RS_\Algebra(f)$, $\chi\in\Algebra$, $0\leq\chi\leq 1$,}
$$
depends on the fact that the class of functions $\chi\in\Algebra$ with values in $[0,1]$ is dense in $L^\infty(X,[0,1])$, w.r.t.~convergence in $\meas$-measure. Here the assumption that $\Algebra$ is a lattice plays a role. 

Because of property (a), the proof of the theorem will be achieved if we show that, for all $f\in D(\Ch)$, one has $f\in D(\Ch_\Algebra)$ and
$|\rmD f|_\Algebra=|\rmD f|$ $\meas$-a.e.~in $X$.  Notice that, because of \eqref{eq:zuri1} and \eqref{eq:zuri11}, one has
$D(\Ch_\Algebra)\subset D(\Ch)$ and $|\rmD f|\leq |\rmD f|_\Algebra$ $\meas$-a.e.~in $X$ for all $f\in D(\Ch_\Algebra)$,
so our main concern will be to prove the converse inclusion and inequality.

For the proof of this statement we can easily reduce ourselves to the case when the support of $\meas$ has finite diameter and $\meas(X)<\infty$. Indeed,
notice that $\Algebra$ contains cut-off functions $h_n:X\to [0,1]$, $n\geq 1$, with $\Lip(h_n)\leq 2$, $h_n\equiv 1$ in $B_{n-1}(\bar x)$ and
$h_n\equiv 0$ in $X\setminus B_n(\bar x)$. 
Now, fix $f\in D(\Ch)$. Denoting by $\meas_n$ the measures $\chi_{B_n(\bar x)}\meas$, and by $\Ch_n$, $|\rmD f|_n$, $\Ch_{\Algebra,n}$,
$|\rmD f|_{\Algebra,n}$ the corresponding relaxed energies and slopes, we obviously have $f\in D(\Ch_n)$ and
$|\rmD f|_n\leq |\rm D f|$ $\meas_n$-a.e.~in $X$. If we are able to show that $f\in D(\Ch_{\Algebra,n})$ and
$|\rmD f|_{\Algebra,n}\leq |\rmD f|_n$ $\meas_n$-a.e.~in $X$ for all $n$, we are then able to find functions 
$f_{k,n}\in \Algebra_\bs$ with $f_{k,n}\to f$ in $L^2(X,\meas_n)$ as $k\to\infty$ 
$$
\limsup_{k\to\infty}\int_X\Lipa ^2(f_{k,n})\di\meas_n\leq\int_X |\rmD f|^2\di\meas_n.
$$
Using the inequality $\Lipa (h g)\leq \Lipa (g)+\Lip(h)g\chi_{\supp (1-h)}$, with $h:X\to [0,1]$ and $\Lip(h)\leq 1$, it is now immediate
to build, by a diagonal argument, $f_{k_n,n}h_n\to f$ in $L^2(X,\meas)$ with
$$
\limsup_{n\to\infty}\int_X\Lipa ^2(f_{k_n,n}h_n)\di\meas=\limsup_{n\to\infty}\int_X\Lipa ^2(f_{k_n,n}h_n)\di\meas_n\leq
\int_X|\rmD f|^2\di\meas.
$$
Since $f_{k_n,n}h_n\in\Algebra_\bs$, this proves that $f\in D(\Ch_\Algebra)$; in addition, by lower semicontinuity, 
we obtain $\int_X|\rmD f|_\Algebra^2\di\meas\leq\int_X|\rmD f|^2\di\meas$,
whence the equality of the relaxed slopes follows.

So, from now on, we assume $\meas(X)<\infty$ and (possibly replacing $X$ by the support of $\meas$) that $(X,\dist)$ has finite
diameter.

\noindent
{\bf Step 2.} (Gradient flow of $\Ch_\Algebra$ and computation of the energy dissipation rate). Thanks to the convexity and lower semicontinuity
properties of $\Ch_\Algebra$ we can define a (possibly nonlinear) continuous semigroup ${\cal P}_t$ in $L^2(X,\meas)$ by considering the
$L^2(X,\meas)$ gradient flow of $\Ch_\Algebra$, as we did in the previous section for $\Ch$. Notice that it is the density of $D(\Ch_\Algebra)$
that enables to the define ${\cal P}_t$ on the whole of $L^2(X,\meas)$.
Specifically, $t\mapsto {\cal P}_tf$ is locally absolutely continuous in $(0,\infty)$ and satisfies
$$
\didi t {\cal P}_tf=\Delta_\Algebra( {\cal P}_tf)\qquad\text{$\leb^1$-a.e.~in $(0,\infty)$.}
$$
Then, using the chain rule, an approximation argument (since $z\mapsto z\log z$ is not Lipschitz) and property (c)
one can prove the entropy dissipation formula
\begin{equation*}
-\didi {t}\int_X f_t\log f_t\di\meas=\int_{\{f_t>0\}}\frac{|\rmD f_t|_\Algebra^2}{f_t}\di\meas
\end{equation*}
with $f_t={\cal P}_tf$, for all $f\in L^2(X,\meas)$ nonnegative with $\int_X f\log f\di\meas<\infty$. 

Notice that the sign condition
is preserved by ${\cal P}_t$, as well as the integral of $f$, namely $\int_X {\cal P}_t f\di\meas=\int_X f\di\meas$. This last
property crucially depends on the finiteness assumption of $\meas$ (the finiteness assumption can be relaxed to \eqref{eq:Grygorian}, still
sufficient for the mass-preserving property). 
Because of this, if $f$ is a probability
density also ${\cal P}_t f$ is a probability density and we shall consider the path of measures
\begin{equation}\label{eq:path}
\mu_t=f_t\meas={\cal P}_tf\meas\qquad t\geq 0.
\end{equation}

\noindent
{\bf Step 3.} (Minimal $2$-weak upper gradient $|\rmD f|_w$ and energy dissipation estimate with $|\rmD f|_w$).
We say that $g\in L^2(X,\meas)$ is a $2$-weak upper gradient if the inequality
\begin{equation}\label{eq:Zurich5}
|f(\gamma(1))-f(\gamma(0))|\leq \int_0^1g(\gamma(s))|\dot\gamma|(s)\di s
\end{equation}
holds $\ppi$-a.e.~for all 2-test plans $\ppi$. Recall that a $2$-test plan is a probability measure $\ppi$ in $C([0,1];X)$
concentrated on $AC^2([0,1];X)$ and satisfying, for some $C=C(\ppi)\geq 0$,
\begin{equation}\label{eq:Zurich2}
(\ev_t)_\#\ppi\leq C\meas\qquad\forall t\in [0,1].
\end{equation}
Because of the non-concentration condition \eqref{eq:Zurich2}, the property of being a $2$-weak upper gradient is
independent of the choice of the representative of $g$ in $L^2(X,\meas)$; furthermore, the class of $2$-weak upper gradients is
a convex closed set and we can identify, as we did for $|\rmD f|_\Algebra$, the element with minimal $L^2(X,\meas)$ 
norm. This distinguished element, the so-called minimal $2$-weak upper gradient, will be denoted by $|\rmD f|_w$. The stability property
of $2$-weak upper gradients (a variant of the so-called Fuglede's lemma), the fact that the asymptotic Lipschitz
constant is a ($2$-weak) upper gradient and \autoref{tapprox} then give 
\begin{equation}\label{eq:Zurich3}
|\rmD f|_w\leq |\rmD f|\qquad\text{$\meas$-a.e.~in $X$ for all $f\in D(\Ch)$.}
\end{equation}
For all probability density $f_0\in L^2(X,\meas)$ having a $2$-weak upper gradient,
by integrating \eqref{eq:Zurich5} with $f=f_0$ and $g=|\rmD f_0|_w$ with respect to a suitable test plan provided by Lisini's superposition
theorem \cite{Li07}, we obtain an estimate on the entropy dissipation rate involving $|\rmD f_0|$ (see \cite[Lemma~5.17]{AmbrosioGigliSavare13}):
\begin{equation}\label{eq:Zurich6}
\int_X f_0\log f_0\di\meas-\int_X f_t\log f_t\di\meas\leq
\frac12\int_0^t\int_{\{f_0>0\}}\frac{|\rmD f_0|_w^2}{f_0^2}f_s\di\meas\di s
+\frac12\int_0^t|\dot\mu|^2(s)\di s
\end{equation}
for all $t>0$. Here $\mu_t=f_t\meas$ is any curve
in $AC^2_\loc([0,\infty);(\Prob(X),W_2))$ and $|\dot\mu|(t)$ is its metric derivative.

Using this sharper energy dissipation estimate we will prove that $f\in D(\Ch_\Algebra)$ for all $f$ having a $2$-weak upper gradient
as well as the inequality $|\rmD f|_w= |\rmD f|_\Algebra$ in the next, final, step. In combination with \eqref{eq:Zurich3} this
provides the converse inequality to \eqref{eq:zuri11} and then the result.

\noindent
{\bf Step 4.} (The equality $|\rmD f|_\Algebra= |\rmD f|_w$). Given $h\in L^2(X,\meas)$ having a 2-weak upper
gradient, we want to prove that $h\in D(\Ch_\Algebra)$ and that $|\rmD h|_\Algebra= |\rmD h|_w$ $\meas$-a.e.~in $X$.
By the local property of this statement, we can consider with no loss of generality only 
$h\in L^\infty(X,\meas)$ with $\int_X h^2\di\meas=1$, $h\geq c>0$ and we set $f=h^2$. Then we consider the
path of measures $\mu_t=f_t\meas$ in \eqref{eq:path}, noticing that still $f_t\geq c^2>0$.

The refined subsolution property \eqref{eq:refined_subs} of \autoref{lemma:subsol} can be used together with the integration by parts formula
(d) to obtain the so-called  Kuwada's lemma (see for instance \cite[Lemma~34]{AmbrosioColomboDiMarino}), which yields
$\mu_t\in AC^2([0,\infty);(\Prob(X),W_2))$ and 
\begin{equation*}
|\dot\mu|(s)\leq \int_X \frac{|\rmD f_s|_\Algebra^2}{f_s}\di\meas\qquad\text{for $\leb^1$-a.e.~$s\in (0,\infty)$.}
\end{equation*}
This, in combination with \eqref{eq:Zurich6} with $f_0=f$, gives
\begin{equation*}
\int_X f\log f\di\meas-\int_X f_t\log f_t\di\meas\leq
\frac12\int_0^t\int_X\frac{|\rmD f|_w^2}{f^2}f_s\di\meas\di s
+\frac12\int_0^t\int_X\frac{|\rmD f_s|_\Algebra^2}{f_s}\di\meas\di s.
\end{equation*}
Hence we deduce
\[
\int_0^t 4\Ch_\Algebra(\sqrt{f_s})\di s=\frac12\int_0^t\int_X\frac{|\rmD f_s|_\Algebra^2}{f_s}\di\meas\di s
\leq\frac12\int_0^t\int_X\frac{|\rmD f|_w^2}{f^2}f_s\di\meas\di s.
\]
Letting $t\downarrow 0$, taking into account the $L^2_\loc$-lower
semicontinuity of $\Ch_\Algebra$ and the fact that $f_t\to f$ in 
$L^2(X,\meas)$, we get $\Ch_\Algebra(h)=\Ch_\Algebra(\sqrt f)\leq\liminf_{t\downarrow
0}\frac1t\int_0^t\Ch_\Algebra(\sqrt{f_s})\di s$. On the other hand, since 
$|\rmD f|_w^2/f^2$ belongs to $L^1(X,\meas)$, the $L^2(X,\meas)$ convergence
of $f_s$ to $f$ (which entails weak$^*$ convergence in $L^\infty(X,\meas)$) yields
$$\int_X\frac{|\rmD f|^2}f\di\meas=\lim_{t\downarrow 0}
\frac 1t \int_0^t\int_X\frac{|\rmD f|_w^2}{f^2}f_s\di\meas\di s.
$$
In summary, we proved that $h\in D(\Ch_\Algebra)$ and that
\[
2\int_X|\rmD h|_\Algebra^2\di\meas\leq
\frac12\int_X\frac{|\rmD f|^2_w}{f}\di\meas=2\int_X|\rmD h|_w^2\di\meas,
\]
where we used the chain rule once once. Taking the inequality $|\rmD h|\leq |\rmD h|_\Algebra$ into account, this
integral inequality proves the coincidence $\meas$-a.e.~in $X$.
\end{proof}

\begin{lemma}[Subsolution property]\label{lemma:subsol}
If $\meas(X)<\infty$, $X$ has finite diameter and $f\in\Cb(X)$, the function
\begin{equation}\label{eq:Hopflax}
Q_t(f):=\inf_{y\in X}f(y)+\frac{1}{2t}\dist^2(x,y)
\end{equation}
is locally Lipschitz in $(0,\infty)\times X$ and
\begin{equation}\label{eq:refined_subs}
\didi{t}Q_tf+\frac{1}{2}|\rmD Q_tf|_\Algebra^2\leq 0 
\qquad\text{$\leb^1\times\meas$-a.e.~in $(0,\infty)\times X$.}
\end{equation}
\end{lemma}
\begin{proof}
Let $D=\{x_k\}$ be as in \eqref{eq:setD}.
Let $n\geq 1$ and define 
$$
Q^n_tf(x):=\min_{1\leq i\leq n}f(x_i)+\frac{1}{2t}\dist^2(x,x_i)
\qquad x\in X,\,\,t>0.
$$
Then $x\mapsto Q^n_t f(x)\in\Algebra$ for all $t>0$, $Q^n_tf$ is locally Lipschitz in $(0,\infty)\times X$ and
$$
\didi{t}Q^n_tf+\frac{1}{2}\Lipa ^2(Q^n_tf)\leq 0 
\qquad\text{$\leb^1\times\meas$-a.e.~in $(0,\infty)\times X$.}
$$
The proof of this inequality is elementary, see for instance \cite[Theorem~14]{AmbrosioColomboDiMarino}. 

In order to obtain \eqref{eq:refined_subs}, we notice that the density of $D$ in $X$ yields that
the family $Q^n_t f$ monotonically converges to $Q_t f$ from above. 
Given $\zeta(t,x)=\chi(t)\psi(x)$, with $\chi\in C^1_c(0,T)$ nonnegative and $\psi\in \Cb(X)$, $\inf\psi>0$, we can pass to the limit in the inequality
$$
\int_0^T\int_X \left(- Q^n_tf\didi{t}\zeta+\frac \zeta 2 \Lipa ^2(Q^n_t f)\right)\di\meas\di t\leq 0
$$
to get
$$
\int_0^T\int_X \left(- Q_tf\didi{t}\zeta +\frac \zeta 2 |\rmD Q_t f)|_\Algebra^2\right)\di\meas\di t\leq 0
$$
which provides, by the arbitrariness of $\zeta$, \eqref{eq:refined_subs}.
In this limiting argument we used the lower semicontinuity property 
$$
\int_X \psi |\rmD g|_\Algebra^2\di\meas\leq\liminf_{n\to\infty}\int_X \psi\Lipa ^2(g_n)\di\meas_n
$$
whenever $g_n\in\Algebra$ and $g_n\to g$ in $L^2(X,\meas)$, which is a simple consequence of the $\meas$-a.e.\ minimality
property of the minimal relaxed slope.
\end{proof}

\end{document}